\documentclass[smallcondensed]{svjour3}
\usepackage[francais,english]{babel}
\usepackage[utf8]{inputenc}
\usepackage[T1]{fontenc}

\smartqed

\usepackage{amsmath,amssymb,amscd,latexsym}
\usepackage{graphicx}
\usepackage{color}
\usepackage{tikz}
\usetikzlibrary{patterns}
\usepackage{mathptmx}

\newcommand{\N}{\mathbb{N}}

\newcommand{\R}{\mathbb{R}}
\newcommand{\PP}{\mathbb{P}}
\newcommand{\EE}{\mathbb{E}}

\newcommand{\e}{\varepsilon}
\newcommand{\D}{\Delta}
\newcommand{\dd}{{\rm d}}

\spnewtheorem{assumption}{Assumption}{\bf}{\it}

\journalname{}

\begin{document}

\title{Simulation of SPDE's for Excitable Media using Finite Elements \thanks{This work was supported by the Agence Nationale de la Recherche through the project MANDy, Mathematical Analysis of Neuronal Dynamics, ANR-09-BLAN-0008-01}}

\author{Muriel Boulakia \and Alexandre Genadot \and Michèle~Thieullen}
\institute{ M. Boulakia \at
              Sorbonne Universités, UPMC Univ Paris 06, UMR 7598, Laboratoire Jacques-Louis Lions, F-75005, Paris, France.\\ 
CNRS, UMR 7598, Laboratoire Jacques-Louis Lions, F-75005, Paris, France.\\ 
INRIA-Paris-Rocquencourt, EPC REO, Domaine de Voluceau, BP105, 78153 Le Chesnay Cedex.\\
\email{boulakia@ann.jussieu.fr}           
           \and
           A. Genadot \at
           Université Paris-Dauphine, UMR 7534, Centre De Recherche en Mathématiques de la Décision, 75775, Paris, France.\\
\email{algenadot@gmail.com}
		\and
M. Thieullen \at
           Sorbonne Universités, UPMC Univ Paris 06, UMR 7599, Laboratoire de Probabilités et Modèles Aléatoires, F-75005, Paris, France.\\
\email{michele.thieullen@upmc.fr} 
}

\date{Received: date / Accepted: date}
\maketitle
\begin{abstract}
In this paper, we address the question of the discretization of Stochastic Partial Differential Equations (SPDE's) for excitable media. Working with SPDE's driven by colored noise, we consider a numerical scheme based on finite differences in time (Euler-Maruyama) and finite elements in space. Motivated by biological considerations, we study numerically the emergence of reentrant patterns in excitable systems such as the Barkley or Mitchell-Schaeffer models.

\end{abstract}


\section{Introduction}

The present paper is concerned with the numerical simulation of Stochastic Partial Differential Equations (SPDE's) used to model excitable cells in order to analyze the effect of noise on such biological systems. Our aim is twofold. The first is to propose an efficient and easy-to-implement method to simulate this kind of models. We focus our work on practical numerical implementation with software used for deterministic PDE's such as FreeFem++ \cite{FF} or equivalent. The second is to analyze the effect of noise on these systems thanks to numerical experiments. Namely, in models for cardiac cells, we investigate the possibility of purely noise induced reentrant patterns such as spiral or scroll-waves as these phenomena are related to major troubles of the cardiac rhythm such as tachyarrhythmia. For numerical experiments, we focus on the Barkley and Mitchell-Schaeffer models, both originally deterministic models to which we add a noise source. \\
Mathematical models for excitable systems may describe a wide range of biological phenomena. Among these phenomena, the most known and studied are certainly the two following ones: the generation and propagation of the nerve impulse along a nerve fiber and the generation and propagation of a cardiac pulse in cardiac cells. For both, following the seminal work \cite{HH52}, very detailed models known as conductance based models have been developed, describing the physio-biological mechanism leading to the generation and propagation of an action potential. These physiological models are quite difficult to handle mathematically and phenomenological models have been proposed. These models describe qualitatively the generation and propagation of an action potential in excitable systems. For instance, the Morris-Lecar model for the nerve impulse and the Fitzhugh-Nagumo model for the cardiac potential. In the present paper, we will consider two phenomenological models: the Barkley and the Mitchell-Schaeffer models.\\
Mathematically, the models that we will consider consist in a degenerate system of Partial Differential Equations (PDE's) driven by a stochastic term, often referred to as noise. More precisely, the model may be written
\begin{equation}\label{eq_model}
\left\{
\begin{array}{ccl}
{\rm d}u&=&[\nu\D u+\frac1\e f(u,v)]{\rm d}t+\sigma {\rm d} W,\\
\dd v&=&g(u,v)\dd t,
\end{array}
\right.
\end{equation}
on $[0,T]\times D$, where $D$ is a regular bounded open set of $\R^2$ or $\R^3$. This system is completed with boundary and initial conditions. $W$ is a colored Gaussian noise source which will be defined more precisely later. System (\ref{eq_model}) is degenerate in two ways: there is no spatial operator such as the Laplacian neither noise source in the equation on $v$. All the considered models have the features of classical stochastic PDEs for excitable systems. The general structure of $f$ and $g$ is also typical of excitable dynamics. In particular, in the models that we will consider, the neutral curve $f(u,v)=0$ for $v$ held fixed is cubic in shape.\\
To achieve our first aim, that is to numerically compute a solution of system (\ref{eq_model}), we work with a numerical scheme based on finite difference discretization in time and finite element method in space. The choice of finite element discretization in space has been directed by two considerations. The first is that this method fits naturally to a general spatial domain: we want to investigate the behavior of solutions to (\ref{eq_model}) on domains with various geometry. The second is that it allows to numerically implement the method with popular software used to simulate deterministic PDE's such as the finite elements software FreeFem++ or equivalent. The discretization of SPDE by finite differences in time and finite elements in space has been considered by several authors in theoretical studies, see for example \cite{DePr09,CYY07,Kr14,KL12,LT11,Wa05}. Other methods of discretization are considered for example in \cite{ANZ98,J09,JR10,KLL10,LT10,Y05}. These methods are based on finite difference discretization in time coupled either to finite difference in space or to the Galerkin spectral method, or to the finite element method on the integral formulation of the evolution equation. We emphasize that we do not consider in this paper a Galerkin spectral method or exponential integrator, that is, roughly speaking, we neither use the spectral decomposition of the solution of (\ref{eq_model}) according to a Hilbert basis of $L^2(D)$ (or another Hilbert space related to $D$) nor the semigroup attached to the linear operator (the Laplacian in (\ref{eq_model})), in order to build our scheme. We only use the variational formulation of the problem in order to fit to commonly used finite elements method for deterministic PDE's, see \cite{BSK81,CL91,GMV12}. Moreover, the present paper is more oriented toward numerical applications than the above cited papers, in the spirit of \cite{Sh05}. In \cite{Sh05}, the author numerically analyzes the effect of noise on excitable systems thanks to a Galerkin spectral method of discretization on the square. In the present paper, we pursue the same objective using the finite element method instead of the Galerkin spectral one. We believe that the finite element method is easier to adapt to various spatial domains. Let us notice that a discretization scheme for SPDE's driven by white noise for spatial domains of dimension greater or equal to 2 may lead to non trivial phenomena, see \cite{HRW12}. Considering colored noises may also be seen as a way to circumvent these difficulties.\\
As is well known, one can consider two types of errors related to a numerical scheme for stochastic evolution equations: the strong error and the weak error. The strong error for the discretization that we consider has been analyzed for one dimensional spatial domains (line segments) in \cite{Wa05}. The weak error for more general spatial domains, of dimension 2 or 3 for example, has been considered in \cite{DePr09}. In the present paper, we recall results about the strong error of convergence of the scheme because we want to numerically investigate pathwise properties of the model.  Working with spatial domains of dimension $d$, as obtained in \cite{Kr14,KL12}, the strong order of convergence of the considered method for a class of linear stochastic equations is twice less than the weak order obtained in \cite{DePr09}. This is what is expected since this same duality between weak and strong order holds for the discretization of finite dimensional stochastic differential equations (SDE's).\\
Our motivation for considering systems such as (\ref{eq_model}) comes from biological considerations. In the cardiac muscle, tachyarrhythmia is a disturbance of the heart rhythm in which the heart rate is abnormally increased. This is a major trouble of the cardiac rhythm since it may lead to rapid loss of consciousness and to death. As explained in \cite{Hi02,JC06}, the vast majority of tachyarrhythmia are perpetuated by a reentrant mechanism. In several studies, it has been observed that deterministic excitable systems of type (\ref{eq_model}) are able to generate sustained reentrant patterns such as spiral or meander, see for example \cite{K80,B90}. We show numerically that reentrant patterns may be generated and perpetuated only by the presence of noise. We perform the simulations on the Barkley model whose deterministic version has been intensively studied in \cite{B90,B92,B94} and the model of Mitchell-Schaeffer which allows to get more realistic shape for the action potential in cardiac cells \cite{MS03,BCFGZ10}. For Barkley model, similar experiments are presented in \cite{Sh05} where Galerkin spectral method is used as simulation scheme on a square domain. In our simulations, done on a square with periodic conditions or on a smoothed cardioid, we observe two kinds of reentrant patterns due to noise: the first may be seen as a scroll wave phenomenon whereas the second corresponds to spiral phenomenon. Both phenomena may be regarded as sources of tachyarrhythmia since in both cases, areas of the spatial domain are successively activated by the same wave which re-enters in the region.\\
All the simulations in the present paper have been performed using the FreeFem++ finite element software, see \cite{FF}. This software offers the advantage to provide the mesh of the domain, the corresponding finite element basis and to solve linear problems related to the finite element discretization of the model on its own. One of the originalities of the present work is to use this software to simulate stochastic PDE's.\\
Let us emphasize that the generic model (\ref{eq_model}) is endowed with a timescale parameter $\e$. The presence of this parameter is fundamental for the observation of traveling waves in the system: $\e$ enforces the system to be either quiescent or excited with a sharp transition between the two states. Moreover, the values of the timescale parameter $\e$ and the strength of the noise $\sigma$ appear to be of first importance to obtain reentrant patterns. This fact is also pointed out by our numerical bifurcation analysis. Let us mention that noise induced phenomena have been studied in \cite{BG06} for finite dimensional system of stochastic differential equations. The theoretical study of slow-fast SPDEs, through averaging methods, has been considered in \cite{Br12,CF09,WR12} for SPDEs.\\
In a forthcoming work, we plan to address the effect of noise on deterministic periodic forcing of the Barkley and Mitchell-Schaeffer models.  We expect to observe as in \cite{JT10} for the one dimensional case, the annihilation by weak noise of the generation of some waves initiated by deterministic periodic forcing. We also want to investigate stochastic resonance phenomena in such a situation. On a theoretical point of view, we intend to derive the strong order of convergence of the discretization method used in the present paper for non-linear equations and systems of equations such as the FitzHugh-Nagumo, Barkley or Mitchell-Schaeffer models but also on simplified conductance based models.\\
The remainder of the paper is organized as follows. In Section \ref{sect_noise}, we begin with the definition of the noise source in system (\ref{eq_model}) and present its finite element discretization. In Section \ref{sect_heat_FHN}, we introduce  a  discretization scheme based on finite element method in space for a stochastic heat equation and we recall the estimates from \cite{Kr14,KL12} for the strong order of convergence. Then we apply the method to Fitzhugh-Nagumo model. In Section \ref{sect_arrhy}, we investigate the influence of noise on Barkley and Mitchell-Schaeffer models. We show that noise may initiate reentrant patterns which are not observable in the deterministic case. We also provide numerical bifurcation diagrams between the noise intensity $\sigma$ and the time-scale $\e$ of the models. At last, some technical proofs are postponed to the Appendix.

\section{Finite element discretization of $Q$-Wiener processes.}\label{sect_noise}

\subsection{Basic facts on $Q$-Wiener processes}\label{sect_Q}

Let $D$ be an open bounded domain of $\R^d$, $d=2$ or $3$, containing the origin and with polyhedral frontier. We denote by $L^2(D)$ the set of square integrable measurable functions with respect to the Lebesgue measure on $\R^d$. Writing $H$ for $L^2(D)$, we recall that $H$ is a real separable Hilbert space and we denote its usual scalar product by $(\cdot,\cdot)$ and the associated norm by $\|\cdot\|$. They are respectively given by
\[
\forall (\phi_1,\phi_2)\in H\times H,\quad (\phi_1,\phi_2)=\int_D \phi_1(x)\phi_2(x){\rm d}x,\quad \|\phi_1\|=\left(\int_D \phi_1(x)^2{\rm d}x\right)^{\frac12}.
\]
Let $Q$ be a non-negative symmetric trace-class operator on $H$. Let us recall the definition of a $Q$-Wiener process on $H$ which can be found in \cite{PeZa07}, Section 4.4, as well as the basic properties of such a process.
\begin{definition}
Let $Q$ be a non-negative symmetric trace-class operator on $H$. There exists a probability space $(\Omega,\mathcal{F},\PP)$ on which we can define a stochastic process $(W^Q_t,~t\in\R_+)$ on $H$ such that
\begin{itemize}
\item For each $t\in\R_+$, $W^Q_t$ is a $H$-valued random variable.
\item $W^Q$ starts at $0$ at time $0$: $W^Q_0=0_{H}$, $\PP$-a.s.
\item $(W^Q_t,~t\in\R_+)$ is a Lévy process, that is, it is a process with independent and stationary increments:
\begin{itemize}
\item Independent increments: for a sequence $t_1,\ldots,t_n$ of strictly increasing times, the random variables $W^Q_{t_2}-W^Q_{t_1},\ldots,W^Q_{t_n}-W^Q_{t_{n-1}}$ are independent.
\item Stationary increments: for two times $s<t$, the random variable $W^Q_{t}-W^Q_{s}$ has same law as $W^Q_{t-s}$.
\end{itemize}
\item $(W^Q_t,~t\in\R_+)$ is a Gaussian process: for any $t\in\R_+$ and any $\phi\in H$, $(W^Q_t,\phi)$ is a real centered Gaussian random variable with variance $t(Q\phi,\phi)$. 
\item $(W^Q_t,~t\in\R_+)$ is a $H$-valued pathwise continuous process, $\PP$-almost surely. 
\end{itemize}
\end{definition}
\noindent We recall the definition of non-negative symmetric linear operator on $H$ admitting a kernel.
\begin{definition}
A non-negative symmetric linear operator $Q:H\to H$ is a linear operator defined on $H$ such that
\[
\forall (\phi_1,\phi_2)\in H\times H, \quad (Q\phi_1,\phi_2)=(Q\phi_2,\phi_1),\quad (Q\phi_1,\phi_1)\geq0. 
\]
Let $q$ be a real valued integrable function on $D\times D$ such that
\begin{align*}
&\forall(x,y)\in \overline{D}\times \overline{D},\quad q(x,y)=q(y,x),\\
&\forall M\in\mathbb{N},\forall x_i,y_j\in \overline{D},\forall a_i\in\mathbb{R},i,j=1,\ldots M,\quad \sum_{i,j=1}^Mq(x_i,y_j)a_ia_j\geq0,
\end{align*}
that is $q$ is symmetric and non-negative definite on $\overline{D}\times \overline{D}$. We say that $Q$ has the kernel $q$ if
\[
\forall \phi\in H,\forall x\in D,\quad Q\phi(x)=\int_D\phi(y)q(x,y){\rm d}y.
\] 
\end{definition}
\noindent Let $Q: H\to H$ be a non-negative symmetric operator with kernel $q$. Then $Q$ is a trace class operator whose trace is given by
\[
{\rm Tr}(Q)=\int_D q(x,x){\rm d}x.
\]
For examples of kernels and basic properties of symmetric non-negative linear operators on Hilbert spaces, we refer to \cite{PeZa07}, Section 4.9.2 and Appendix A. Let us now state clearly our assumptions on the operator $Q$.
\begin{assumption}\label{cond_C}
The operator $Q$ is a non-negative symmetric operator with kernel $q$ given by
\[
\forall (x,y)\in \overline{D}\times \overline{D},\quad q(x,y)=C(x-y),
\]
where $C$ belongs to $\mathcal{C}^{2}(\overline{D})$ and is an even function on $\overline{D}$ satisfying:
\[
\forall M\in\mathbb{N},\forall x_i,y_j\in \overline{D},\forall a_i\in\mathbb{R},i,j=1,\ldots M,\quad \sum_{i,j=1}^MC(x_i-y_j)a_ia_j\geq0.
\]
Particularly, $\nabla C(0)=0$ and $x\mapsto \frac{C(x)-C(0)}{|x|^2}$ is bounded on a neighborhood of zero.
\end{assumption}
\noindent For $x\in \overline{D}$ and $t\in\R_+$, one can show, see \cite{PeZa07}, Section 4.4, that we can define $W^Q_t$ at the point $x$ such that the process $(W^Q_t(x),(t,x)\in\R_+\times \overline{D})$ is a centered Gaussian process with covariance between the points $(t,x)$ and $(s,y)$ given by
\[
\EE\left(W^Q_t(x)W^Q_s(y)\right)=t\wedge s~q(x,y).
\]
In this case, the correlations in time are said to be white whereas the correlations in space are colored by the kernel $q$. 
\begin{proposition}\label{Cont_Bruit}
Under Assumption \ref{cond_C}, the process $(W^Q_t(x),(t,x)\in\R_+\times \overline{D})$ has a version with continuous paths in space and time.
\end{proposition}
\begin{proof}
This is an easy application of the Kolmogorov-Chentsov test, see \cite{DaZa92} Chapter 3, Section 3.2. Note that the regularity of $C$ is important to get the result. \qed
\end{proof}
\begin{remark}\label{Diff_Bruit2}
The more the kernel $q$ smooth is, the more the Wiener process regular is. For example, let $\vec \iota\in \R^d$ and define $f(x)=\vec{\iota}\cdot{\rm Hess}C(x)\vec{\iota}$ for $x\in \overline{D}$. Suppose that $f$ is a twice differentiable function. Then, one can show that there exists a probability space on which $(W^Q_t,~t\in\R_+)$ is twice differentiable in the direction $\vec \iota$.
\end{remark}
\begin{remark}
Let us assume that there exists a constant $\alpha$ and a (small) positive real $\delta$ such that
\[
\forall y\in \overline{D},\quad |C(0)-C(y)|\leq \alpha|y|^{2+\delta}. 
\]
Using the Kolmogorov-Chentsov continuity theorem, one can show that the process \[(W^Q_t(x),(t,x)\in \R_+\times \overline{D})\] has a modification which is $\gamma_1$-Hölder in time for all $\gamma_1\in\left(0,\frac12\right)$ and $\gamma_2$-Hölder in space for all $\gamma_2\in\left(0,1+\frac{\delta}{2}\right)$. Thus if $\delta>0$, by Rademacher theorem, for $\gamma_2=1$, this version is almost everywhere differentiable on $\overline{D}$. This is another way to obtain regularity in space for $W^Q$ without using another probability space.
\end{remark}

\subsection{Finite element discretization}\label{section_fe}

In this part, we assume that $Q$ satisfies Assumption \ref{cond_C}. Let us present our approximation of the $Q$-Wiener process $W^Q$. We begin with the discretization of the domain $D$. Let $\mathcal{T}_h$ be a family of triangulations of the domain $D$ by triangles ($d=2$) or tetrahedra ($d=3$). The size of $\mathcal{T}_h$ is given by
\[
h=\max_{T\in\mathcal{T}_h} h(T),
\]
where $h(T)=\max_{x,y\in T}|x-y|$ is the diameter of the element $T$. We assume that there exist a positive constant $\rho$ such that
\begin{equation}\label{hyp_triangle}
\forall h>0,~\forall T\in\mathcal{T}_h,\quad\exists x\in T,\quad T\subset {\rm B}(x,\rho h),
\end{equation}
where ${\rm B}(x,r)$ stands for the euclidean ball centered at $x$ with radius $r$. We assume further that this triangulation is regular as in Figure \ref{Fig_mesh} (see \cite{RaTh83} p. 108 for a definition) where a triangulation is displayed and the property (\ref{hyp_triangle}) is illustrated.
\begin{figure}
\begin{center}
\begin{tabular}{cc}
\includegraphics[width=5cm]{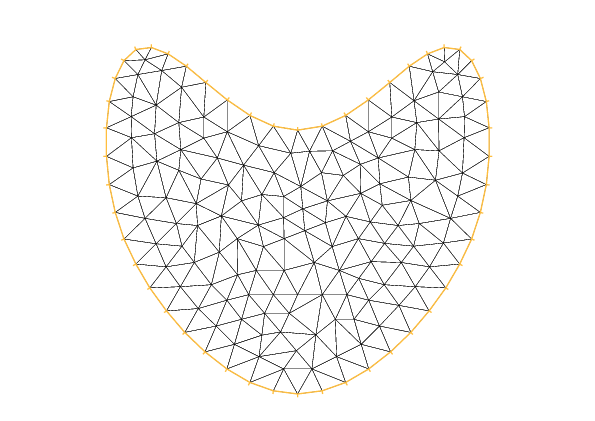}&\begin{tikzpicture}
\draw (1,1)--(2,3)--(3,1)--(1,1);
\draw[thick,blue] (2,1.75) circle (1.25);
\draw[thick] (2,1.75)--node[midway,above]{$\rho h$}(1,2.5);
\draw[white] (0,0)--(3,0);
\end{tikzpicture}
\end{tabular}
\end{center}
\caption{Meshing (triangulation) of a domain and illustration of (\ref{hyp_triangle})}\label{Fig_mesh}
\end{figure}
\noindent In the present work, we consider two kinds of finite elements: the Lagrangian P0 and P1 finite elements. However, the  method could be adapted to other finite elements. The basis associated to the P0 finite element method is
\[
\mathcal{B}_{0,\mathcal{T}_h}=\{1_T,T\in\mathcal{T}_h\},
\]
where the function $1_T$ denotes the indicator function of the element $T$. Let $\{P_i, 1\leq i\leq N_h\}$ be the set of all the nodes associated to the triangulation $\mathcal{T}_h$.  The basis for the P1 finite element method is given by
\[
\mathcal{B}_{1,\mathcal{T}_h}=\{\psi_i,1\leq i\leq N_h\},
\]
where $\psi_i$ is the continuous piecewise affine function on $D$ defined by $\psi_i(P_j)=\delta_{ij}$ (Kronecker symbol) for all $1\leq i,j\leq N_h$.
\begin{definition}\label{Def_app}
The P0 approximation of the noise $W^Q$ is given for $t\in\R_+$ by
\begin{equation}\label{noise_P0}
W^{Q,h,0}_t=\sum_{T\in\mathcal{T}_h}W^Q_t(g_T)1_T,
\end{equation}
where $g_T$ is the center of gravity of $T$. The P1 approximation is
\begin{equation}\label{noise_P1}
W^{Q,h,1}_t=\sum_{i=1}^{N_h}W^Q_t(P_i)\psi_i.
\end{equation}
We will also consider the following alternative choice for the P0 discretization
\begin{equation}\label{noise_P0_CYY}
W^{Q,h,0_a}_t=\sum_{T\in\mathcal{T}_h}\frac{1}{|T|}(W^Q_t,1_T)1_T.
\end{equation}
\end{definition}
\noindent $W^{Q,h,0_a}$ corresponds to an orthonormal projection on P0. These approximations are again Wiener processes as stated in the following proposition.
\begin{proposition}
For $i\in\{0,0_a,1\}$ the stochastic processes $(W^{Q,h,i}_t,~t\in\R_+)$ are centered $Q^{h,i}$-Wiener processes where, for $\phi\in H$
\begin{eqnarray*}
Q^{h,0}\phi&=&\sum_{T,S\in\mathcal{T}_h}(1_T,\phi)q(g_T,g_S)1_S,\\
Q^{h,0_a}\phi&=&\sum_{T,S\in\mathcal{T}_h}(1_T,\phi)\frac{(Q1_T,1_S)}{|T||S|}1_S\\
\end{eqnarray*}
and
\[
Q^{h,1}\phi=\sum_{i,j=1}^{N_h}(\psi_i,\phi)q(P_i,P_j)\psi_{j}.
\]
\end{proposition}
\begin{proof}
The fact that for $i\in\{0,0_a,1\}$ the stochastic processes $(W^{Q,h,i}_t,~t\in\R_+)$ are Wiener processes is a direct consequence of their definition as linear functionals of the Wiener process $(W^Q_t,~t\in\R_+)$, see Definition \ref{Def_app}. The corresponding covariance operators are obtained by computing the quantity
\[\EE((W^{Q,h,i}_1,\phi_1)(W^{Q,h,i}_1,\phi_2))
\]
for $i\in\{0,0_a,1\}$ and $\phi_1,\phi_2\in H$ (the details are left to the reader). \qed
\end{proof}
\noindent The P0 approximation (\ref{noise_P0_CYY}) of the noise has been considered for white noise in dimension 2 in \cite{CYY07}. White noise corresponds to $Q={\rm Id}_{H}$. In the white noise case, the associated Wiener process is not at all regular in space (the trace of $Q$ is infinite in this case). In the present paper, we work with trace class operators and thus with noises which are regular in space. Notice that discretization schemes for SPDE driven by white noise for spatial domains of dimension greater or equal to 2 may lead to non trivial phenomena. In particular, usual schemes may not converge to the desired SPDE, see \cite{HRW12}.
\begin{theorem}[A global error]\label{prop_glob_err}
For any $\tau\in\R_+$ and $i\in\{0,0_a,1\}$ we have
\[
\EE\left(\sup_{t\in[0,\tau]}\|W^Q_t-W^{Q,h,i}_t\|^2\right)\leq K\tau h^2
\]
where $K$ may be written $K=\tilde{K}\|{\rm Hess}~C\|_{\infty}$ for some constant $\tilde{K}$ which only depends on $|D|$.
\end{theorem}
\begin{proof}
The proof is postponed to Appendix \ref{app_bgt_1}.\qed
\end{proof}
\noindent Let us comment the above result. Let us take, as it will be the case in the numerical experiments, the following special form for the kernel
\[
\forall x\in D,\quad C_\xi(x)=\frac{a}{\xi^2}e^{-\frac{b}{\xi^2}|x|^2}
\]
for three positive real numbers $a,b,\xi$, where $D$ is a bounded domain in dimension $2$. This is a so-called Gaussian kernel. To a particular $\xi$-dependent kernel $C_\xi$, we associate the corresponding  $\xi$-dependent covariance operator $Q_\xi$. Then, according to Theorem \ref{prop_glob_err}, in this particular case, we see that
\[
\EE\left(\sup_{t\in[0,\tau]}\|W^{Q_\xi}_t-W^{Q_\xi,h,i}_t\|^2\right)={\rm O}\left(\tau \frac{h^2}{\xi^4}\right)
\]
for any $\tau\in\R_+$. Thus, when $\xi$ goes to zero, this estimation becomes useless since the right hand-side goes to infinity. In fact, when $\xi$ goes to zero, $C_\xi$ converges in the distributional sense to a Dirac mass, and $W^{Q_\xi}$ tends to a white noise which is, as mentioned before, an irregular process. In particular, the white noise does not belong to $H$ and this is why our estimation is no longer useful in this case. The same phenomenon occurs for any bounded domain of dimension $d\geq2$. Let us mention that for white noise acting on steady PDEs and on particular domains (square and disc), the error considered in Theorem \ref{prop_glob_err} have been studied in \cite{CYY07}: the regularity of the colored noised improved these estimates in our case. We also remark that the proof of Theorem \ref{prop_glob_err} does not rely on the regularity of the functions of the finite element basis of P0 or P1 here. The key points are that $C$ is smooth enough, even and that $\sum_i\phi_i=1$, where $\{\phi_i\}$ corresponds to the finite element basis. \\
\noindent To conclude this section, we display some simulations. In Figure \ref{Fig_bruit_P1} we show simulations of the noise $W^{Q_\xi}_1$ with covariance kernel defined by
\begin{equation}\label{kernel_xi}
\forall (x,y)\in D\times D,\quad q_\xi(x,y)=C_\xi(x-y)=\frac{1}{4\xi^2}e^{-\frac{\pi}{4\xi^2}|x-y|^2},
\end{equation}
where $\xi>0$. We use the same kernel as in \cite{Sh05} for comparison purposes. As already mentioned, when $\xi$ goes to zero, the considered colored noise tends to a white noise. On the contrary, when $\xi$ increases, the correlation between two distinct areas increases as well. This property is illustrated in Figure \ref{Fig_bruit_P1}. In other words, $\xi$ is a parameter which allows to control the spatial correlation.\\
In these simulations, we have discretized $W^{Q_\xi,h,1}$ with the P1 discretization which reads 
\[
W^{Q_\xi,h,1}_1=\sum_{i=1}^{N_h}W^{Q_\xi}_1(P_i)\psi_i.
\]
We remark that the family $\{W^{Q_\xi}_1(P_i),1\leq i\leq N_h\}$ is a centered Gaussian vector with covariance matrix $(q_\xi(P_i,P_j))_{1\leq i,j\leq N_h}$. Using some basic linear algebra, it is not difficult to simulate a realization of this vector and to project it on the P1 finite element basis to obtain Figure \ref{Fig_bruit_P1}.
\begin{figure}
\begin{center}
\begin{tabular}{cc}
\includegraphics[width=3.5cm]{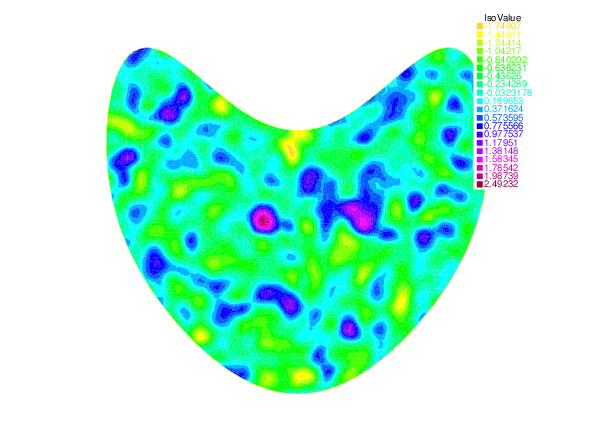}&\includegraphics[width=3.5cm]{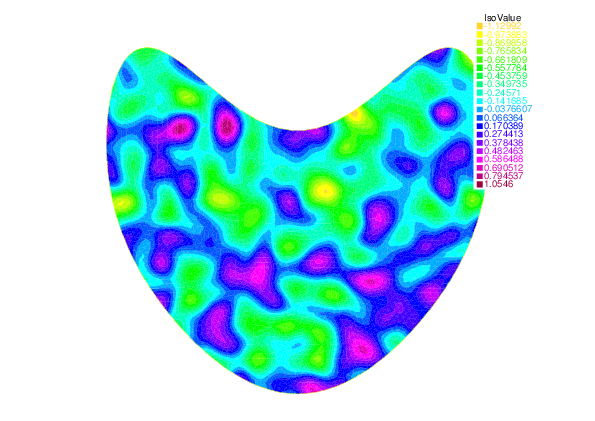}\\
\includegraphics[width=3.5cm]{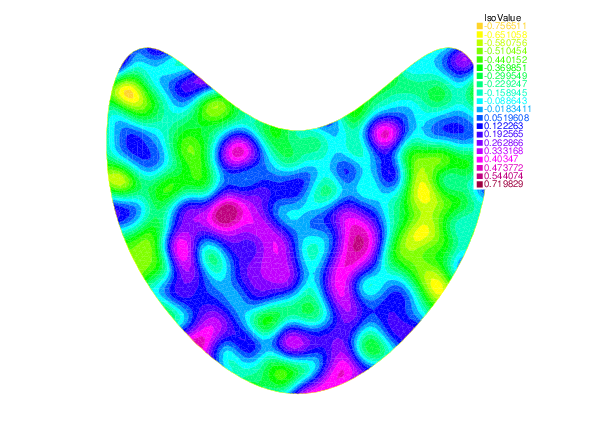}&\includegraphics[width=3.5cm]{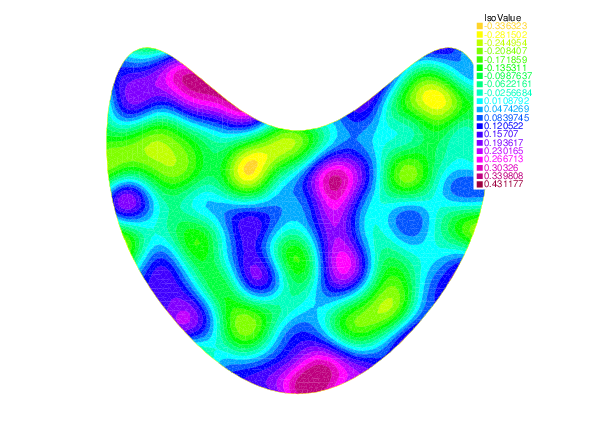}
\end{tabular}
\caption{Simulations of $W^{Q_\xi}_1$ with $\xi=1,1.5,2,3$ with P1 finite elements.}\label{Fig_bruit_P1}
\end{center}
\end{figure}

\noindent We now propose a log-log graph to illustrate the estimate of Theorem \ref{prop_glob_err}. In the case where $\overline{D}$ is the square $[0,1]\times[0,1]$, let us consider the kernel $q$ given by:
\[
\forall (x,y)\in\overline{D}\times \overline{D},\quad q(x,y)=f_{k_0p_0}(x)f_{k_0p_0}(y),
\]
where for two given integers $k_0,p_0\geq1$, $f_{k_0p_0}(x)=2\sin(k_0\pi x_1)\sin(p_0\pi x_2)$ (if $x=(x_1,x_2)\in\overline{D}$). Then, the covariance operator is given by $Q\phi=(\phi,f_{k_0p_0})f_{k_0p_0}$ for any $\phi\in L^2(\overline{D})$. Moreover, one can show that
\[
\forall t\geq0,\quad W^Q_t=\beta_t f_{k_0p_0}
\]
for some real-valued Brownian motion $\beta$. All the calculations are straightforward in this setting. Suppose, in the finite element setting, that the square $\overline{D}$ is covered by $2N^2$ triangles, $N\in\N$. For any $N\in\N$, we denote by $W^{Q,N,0}_{1}$ the P0 approximation of $W^Q_1$ given by (\ref{noise_P0}). We show in Figure \ref{fig_P0_err} the log-log graph of the (discrete) function: 
\[
N\mapsto \mu_N=\EE(\|W^Q_1-W^{Q,N,0}_{1}\|^2).
\]
According to Theorem \ref{prop_glob_err}, we should have $\mu_N={\rm O}\left(\frac{1}{2N^2}\right)$. This result is recovered numerically in Figure \ref{fig_P0_err}. 
\begin{figure}
\begin{center}
\includegraphics[width=3.5cm]{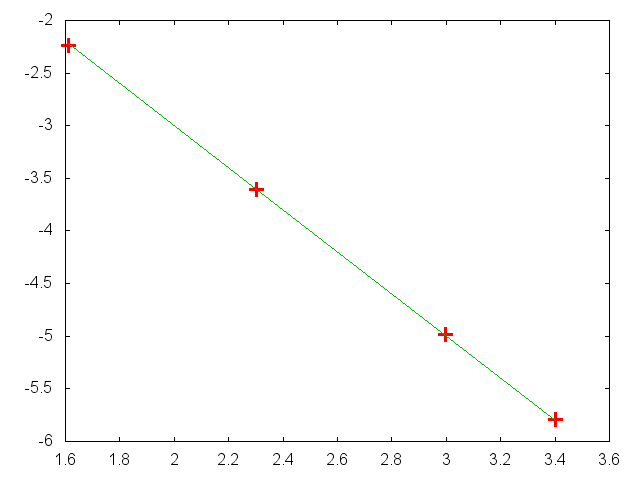}
\end{center}
\caption{Log-log graph of $N\mapsto \mu_N$ for $N=5,10,20,30$. In green is the comparison with a line of slope $-2$.}\label{fig_P0_err}
\end{figure}

\section{Space-time numerical scheme}\label{sect_heat_FHN}

In this section, we first present our numerical scheme. The considered space-time discretization is based on the Euler scheme in time and on finite elements in space. In this section and in the next section, we will use the following notations. Let us fix a time horizon $T$. For $N\geq1$ we define a time step $\D t=\frac{T}{N}$ and denote by $(u_n,v_n)_{0\leq n\leq N}$ a sequence of approximations of the solution of (\ref{eq_model}) at times $t_n=n\D t$, $0\leq n\leq N$. The scheme, semi-discretized in time, is based on the following variational formulation, for $n\in\{0,\ldots,N-1\}$:
\begin{equation}\label{scheme_ex_weak}
\left\{
\begin{array}{ccl}
\left(\frac{u_{n+1}-u_n}{\D t},\psi\right)+\kappa(\nabla u_{n+1},\nabla\psi)&=&\frac{1}{\e}(f_n,\psi)+\frac{\sigma}{\D t}(W^Q_{n+1}-W^Q_n,\psi),\\
\left(\frac{v_{n+1}-v_n}{\D t},\phi\right)&=&(g_n,\phi)
\end{array}
\right.
\end{equation}
for $\psi$ and $\phi$ in appropriate spaces of test functions. Here, $f_n$ and $g_n$ correspond to approximations of the reaction terms $f$ and $g$ in (\ref{eq_model}). The way we compute $f_n$ and $g_n$ is detailed in the sequel for each considered model. $W^Q_{n}$ is an appropriate approximation of $W^Q_{t_n}$ based on one of the discretizations proposed in Definition \ref{Def_app}. \\
In Subsection \ref{sect_heat}, we consider the strong error in the case of a linear stochastic partial differential equation driven by a  colored noise to study the accuracy of the finite element discretization. We recall the estimates on the strong order of convergence for such a numerical scheme obtained in \cite{Kr14} and we numerically illustrate this result for an explicit example. In Subsection \ref{sect_FHN}, we present in more details the scheme for the Fitzhugh-Nagumo model with a colored noise source since this model is one of the most used phenomenological models in cardiac electro-physiology, see the seminal work \cite{Fi69} and the review \cite{LiGaNeSc04}.

\subsection{Linear parabolic equation with additive colored noise}\label{sect_heat}

Let us consider the following linear parabolic stochastic equation on $(0,T)\times D$
\begin{equation}\label{eq_l}
\left\{\begin{array}{ccl}
\dd u_t&=& A u_t\dd t+\sigma \dd W^{Q}_t,\\
u_0&=&\zeta .
\end{array}\right.
\end{equation}
Remember that $H=L^2(D)$ is a separable Hilbert space with the scalar product and the corresponding norm respectively denoted by $(\cdot,\cdot)$ and $\|\cdot\|$. We assume that $W^Q$ is a $Q$-Wiener process with an operator $Q$ which satisfies Assumption \ref{cond_C}. We impose the following condition on the operator $A$ in (\ref{eq_l}).
\begin{assumption}\label{ass_A}
The operator $-A$ is a positive self-adjoint linear operator on $H$ whose domain is dense and compactly embedded in $H$. 
\end{assumption}
\noindent It is well known that the spectrum of $-A$ is made up of an increasing sequence of positive eigenvalues $(\lambda_i)_{i\geq1}$. The corresponding eigenvectors $\{w_i,i\geq1\}$ form a Hilbert basis of $H$.\\

\noindent The domain of $(-A)^{\frac12}$ is the set
\[
\left\{u=\sum_{i\geq1}(u,w_i)w_i,\quad \sum_{i\geq1}\lambda_i(u,w_i)^2<\infty\right\},
\]
that we denote here by $V$. It is continuously and densely embedded in $H$. The $V$-norm is given by $|u|=\sqrt{-(Au,u)}$ for all $u\in V$. We define a coercive continuous bilinear form $a$ on $V\times V$ by
\[
a(u,v)=-(Au,v).
\]
Let $\zeta$ be a $V$-valued random variable. The following proposition states that problem (\ref{eq_l}) is well posed.
\begin{proposition}\label{mild_sol}
Equation (\ref{eq_l}) has a unique mild solution:
\[
u_t=e^{At}\zeta+\sigma\int_0^t e^{A(t-s)}\dd W^Q_s,
\]
Moreover $u$ is continuous in time and $u_t\in V$ for all $t\in[0,T]$, $\PP$-a.s.
\end{proposition}
\begin{proof}
This result is a direct consequence of Theorem 5.4 of \cite{DaZa92}, Assumptions \ref{cond_C} and \ref{ass_A}. \qed
\end{proof}
For $h>0$, let $V_h$ be a finite dimensional subset of $V$ with the property that for all $v\in V$, there exists a  sequence of elements $v_h\in V_h$ such that $\lim_{h\to 0}\|v-v_h\|=0$. For an element $u$ of $V$, we introduce its orthogonal projection on $V_{h}$ and denote it $\Pi_hu$. It is defined in a unique way by
\begin{equation}\label{def_proj}
\Pi_hu\in V_h\quad{\rm and}\quad\forall v_h\in V_{h},\quad a(\Pi_h u-u,v_h)=0.
\end{equation}
Let $I_h$ be the dimension of $V_h$. Notice that there exists a basis $(w_{i,h})_{1\leq i\leq I_h}$ of $V_h$ orthonormal in $H$ with the following property: for each $1\leq i\leq I_h$, there exists $\lambda_{i,h}$ such that
\[
\forall v_h\in V_h,\quad a(v_h,w_{i,h})=\lambda_{i,h}(v_h,w_{i,h}),
\]
(see \cite{RaTh83}, Section 6.4). The family $(\lambda_{i,h})_{1\leq i\leq I_h}$ is an approximating sequence of the family of eigenvalues $(\lambda_i)_{i\geq1}$ so that
\[
\lambda_{i,h}\geq \lambda_i,\quad\forall 1\leq i\leq I_h.
\]
We study the following numerical scheme to approximate equation (\ref{eq_l}) defined recursively as follows. For $u_0$ given in $V_h$, find $(u^h_n)_{0\leq n\leq N}$ in $V_h$ such that for all $n\leq N-1$
\begin{equation}\label{scheme_eq_l}
\left\{\begin{array}{rcl}
\frac{1}{\D t}(u^h_{n+1}-u^h_n,v_h)+a(u^h_{n+1},v_h)&=&\frac{\sigma}{\D t}(W^{Q,h}_{n+1}-W^{Q,h}_n,v_h)\\
u^h_0&=&u_0
\end{array}\right.
\end{equation}
for all $v_h\in V_h$ where $W^{Q,h}_n$ is an appropriate approximation of $W^Q_{n \D t}$ in $V_h$.
The approximation error of the scheme can be written as the sum of two errors.
\begin{definition}
The discrete error introduced by the scheme (\ref{scheme_eq_l}) is defined by $E^h_n=e^h_n+p^h_n$ where
\begin{equation}
e^h_n=u^h_n-\Pi_h u_{t_n},\quad p^h_n=\Pi_h u_{t_n}-u_{t_n}
\end{equation}
for $0\leq n\leq N$.
\end{definition}
For $n\in\{0,\ldots,N\}$, the error $e^h_n$ is the difference between the approximated solution given by the scheme and the elliptic projection on $V_{h}$ of the exact solution at time $n\D t$. The error $p^h_n$ is the difference between the exact solution and its projection on $V_{h}$ at time $n\D t$. \\

\noindent In order to give explicit bounds for the error defined above, let us choose our approximation in $H$ specifically. This imposes to choose the space $V$ explicitly. Assume that the operator $A$ is such that $V= H^1_0(D)$ and that $V_h$ is a space of P1 finite elements, see Section \ref{section_fe}. In this P1 case, for $n\in\{0,\ldots,N\}$, we set $W^{Q,h}_n=W^{Q,h,1}_{n\D t}$ defined by Definition \ref{Def_app}.
The situation considered in the present section is also the one studied in example 3.4 and section 7 of \cite{Kr14}. We have the following bound for the numerical error coming from such a numerical scheme. 
\begin{theorem}[Example 3.4 and Corollary 7.2 in \cite{Kr14}]\label{cor:main}
Let us assume that Assumptions \ref{cond_C} and \ref{ass_A} are satisfied. Moreover, assume that we are in the P1 case: for $n\in\{0,\ldots,N\}$, $W^{Q,h}_n=W^{Q,h,1}_{n\D t}$ defined by Definition \ref{Def_app}. Then, there exists $\D t_0>0$ such that for all $n\in\{1,\ldots,N\}$ and $\D t\in[0,\D t_0]$
\begin{equation}\label{so}
\sqrt{\EE(\|E^h_n\|^2)}\leq K(h+\sqrt{\D t}),
\end{equation}
where $K$ is a constant depending only on $T$ and $|D|$.
\end{theorem}

\noindent Let us recall the weak order of convergence of the considered scheme obtained in \cite{DePr09} but under weaker assumptions. Since $C$ is a twice differentiable even function on $D$, $\Delta C$ is a bounded function on $D$ and therefore, according to \cite{DePr09} Theorem 3.1, for any bounded real valued twice differentiable function $\phi$ on $L^2(D)$, there exists a constant $K$ depending only on $T$ such that
\begin{equation}\label{wo}
|\EE(\phi(u^h_N))-\EE(\phi(u_T))|\leq K(h^{2\gamma}+\D t^\gamma)
\end{equation}
for a given $\gamma<1$. In our situation, it is more natural to consider the strong error since we study pathwise behavior. For the method that we consider, estimates for the strong error have been obtained for one dimensional spatial domains and white noise in \cite{Wa05}. Many papers exist for finite dimensional systems. The estimate of Theorem \ref{cor:main} lies in between these two types of studies. The noise is colored but the spatial domain may be of any dimension. Notice that the order of weak convergence (\ref{wo}) is twice the order of strong convergence (\ref{so}), as for finite dimensional stochastic differential equations.\\

\noindent In the end of this subsection, we illustrate this error estimate in a simple situation. We consider the domain $D=(0,l)\times (0,l)$ for $l>0$. We set $A=\Delta$ and $\mathcal{D}(A)=H^2(D)\cap H^1_0(D)$, thus $V=H^1_0(D)$.  That is we consider the equation:
\begin{equation}\label{heat}
\left\{
\begin{array}{ccl}
{\rm d}u_t&=&\D u_t{\rm d}t+\sigma \dd W^{Q}_t,\quad{\rm in}~D,\\
u_t&=&0,\quad{\rm on}~\partial D,\\
u_0&=&0,\quad{\rm in}~D
\end{array}
\right.
\end{equation}
for $t\in\R_+$.  $W^{Q}$ is a $L^2(D)$-valued $Q$-Wiener process defined by Definition \ref{cond_C} and $Q$ satisfies Assumption \ref{cond_C}. The initial condition is zero, hence the solution of the corresponding deterministic equation, without noise, is simply zero for all time. Following Proposition \ref{mild_sol}, equation (\ref{heat}) has a unique mild solution such that $u_t\in H^1_0(D)$ for all $t\in[0,T]$, $\PP$-almost surely. Moreover, $u$ has a version with time continuous paths and such that, for any time $T>0$:
\[
\sup_{t\in[0,T]}\EE(\|u_t\|^2_{H^1_0(D)})<\infty.
\]
\noindent We denote by $(e^{\Delta t},~t\geq0)$ the contraction semigroup associated to the operator $\D$. The mild solution to equation (\ref{heat}) is defined as the following stochastic convolution
\[
u_t=\sigma\int_0^te^{\D (t-s)}\dd W^{Q}_s
\]
for $t\in\R_+$, $\PP$-almost-surely. In order to compute the expectation of the squared norm of $u$ in $L^2(D)$ analytically and also as precisely as possible numerically, we define the Hilbert basis $(e_{kp},k,p\geq1)$ of $L^2(D)$ which diagonalizes the operator $\D$ defined on $\mathcal{D}(A)$. For $k,p\geq1$ and $(x,y)\in D$
\[
e_{kp}(x,y)=\frac{2}{l}\sin\left(\frac{k\pi}{l}x\right)\sin\left(\frac{p\pi}{l}y\right).
\]
A direct computation shows that $\D e_{kp}=-\lambda_{kp}e_{kp}$ where $\lambda_{kp}=\frac{\pi^2}{l^2}(k^2+p^2)$. In the basis $(e_{kp},k,p\geq1)$ of $L^2(D)$, the semigroup $(e^{\D t},~t\geq0)$ is given by
\[
e^{\D t}\phi=\sum_{k,p\geq1}e^{-\lambda_{kp} t}(\phi,e_{kp})e_{kp}
\]
for $t\in\R_+$ and $\phi\in L^2(D)$. Then for any $t\in\R_+$ (c.f. Proposition 2.2.2 of \cite{DaPa04})
\[
\EE(\|u_t\|^2)=\sigma^2\int_0^t{\rm Tr}\left(e^{2\D s}Q\right)\dd s=\sigma^2\sum_{k,p\geq1}\frac{1-e^{-2\lambda_{kp} t}}{2\lambda_{kp}}(Q e_{kp},e_{kp}).
\]
In the sequel, we write $\Gamma_t=\EE(\|u_t\|^2)$. The above series expansion can then be implemented and we can compare this result with $\EE(\|u^h_n\|^2)$ which is computed thanks to Monte-Carlo simulations. The Monte-Carlo simulation of $\EE(\|u^h_n\|^2)$ consists in considering $(u^{h,p}_n)_{1\leq p\leq P}$, $P\in\N$ a sequence of independent realizations of the scheme (\ref{scheme_eq_l}) and define
\begin{equation}
\Gamma^{(P)}_{n\D t}=\frac{1}{P}\sum_{p=1}^P\|u^{h,p}_n\|^2,
\end{equation}
the approximation of $\Gamma$ at time $n\D t$, $n\in\{0,\ldots,N\}$. We denote also by $\Gamma^{(P)}$ the continuous piecewise linear version of $\Gamma$. Figure \ref{Fig_var} displays numerical simulations of the processes $(\Gamma_t,~t\in\R_+)$ and $(\Gamma^{(P)}_t,~t\in\R_+)$. The simulations are done with $l=80$. Moreover the domain is triangulated with $5000$ triangles giving a space step of about $h=0.64$ and a number of vertices's of about $2600$. For this simulation, we choose $P=40$ which is not big but $\Gamma^{(40)}$ matches quite well with its corresponding theoretical version $\Gamma$, as expected by the law of large numbers. We remark also that for the same spatial discretization of the domain $D$, there is no particular statistical improvement to choose the P1 finite element basis instead of the P0.
\begin{figure}
\begin{center}
\begin{tabular}{cc}
\includegraphics[width=5cm]{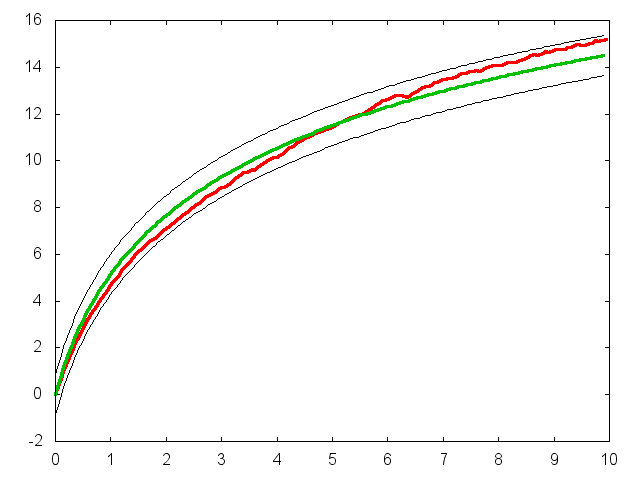}&\includegraphics[width=5cm]{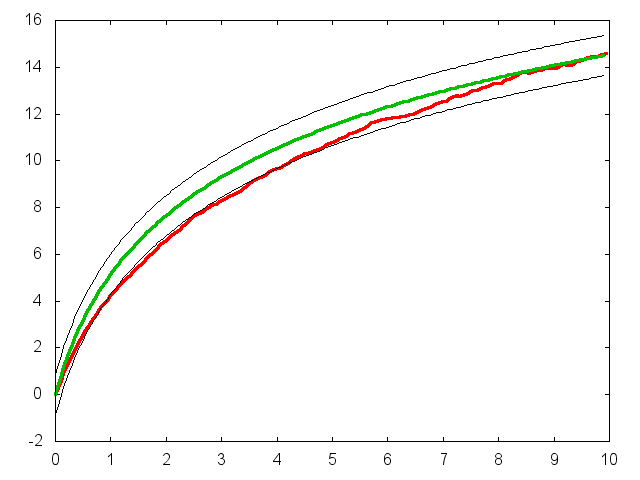}
\end{tabular}
\caption{Simulations of $(\Gamma_t,~t\in[0,10])$ in green and its approximation $\Gamma^{(40)}$ in red computed with: the P0 (on the left) and P1 (on the right) approximations of the noise. For the simulation we choose the coefficient of correlation $\xi=2$ with the kernel $q_\xi$ defined by (\ref{kernel_xi}). The intensity of the noise is $\sigma=0.15$. The time step is $0.05$ whereas the space step is about $0.64$. The two black curves are respectively $\Gamma$ plus, respectively minus, the error introduced by the scheme which is expected to be of order $\sqrt{\D t}+h$ equals here to $\sqrt{0.05}+0.64$.}\label{Fig_var}
\end{center}
\end{figure}

\subsection{Space-time discretization of the Fitzhugh-Nagumo model}\label{sect_FHN}

We write the scheme for the Fitzhugh-Nagumo which is a widely used model of excitable cells, see \cite{Fi69,LiGaNeSc04}. The stochastic Fitzhugh-Nagumo model, abbreviated by FHN model in the sequel, consists in the following $2$-dimensional system
\begin{equation}\label{FHN_BGT}
\left\{
\begin{array}{ccl}
{\rm d}u&=&[\kappa\D u+\frac1\e\left(u(1-u)(u-a)-v\right)]{\rm d}t+\sigma \dd W^Q,\\
\dd v&=&[u-v]\dd t,
\end{array}
\right.
\end{equation}
on $[0,T]\times D$. In the above system, $\kappa>0$ is a \emph{diffusion} coefficient, $\e>0$ a \emph{time-scale} coefficient, $\sigma>0$ the intensity of the noise and $a\in(0,1)$ a parameter. $W^Q$ is a $Q$-Wiener process satisfying Assumption \ref{cond_C}. System (\ref{FHN_BGT}) must be endowed with initial and boundary conditions. We denote by $u_0$ and $v_0$ the initial conditions for $u$ and $v$. Moreover we assume that $u$ satisfies zero Neumann boundary conditions:
\begin{equation}\label{bound_cond}
\forall t\in[0,T],\quad \frac{\partial u_t}{\partial \vec{n}}=0,\quad\text{ on }\partial D,
\end{equation}
where $\partial D$ denotes the boundary of $D$ and $\vec{n}$ is the external unit normal to this boundary. Noisy FHN model and especially, FHN with white noise, have been extensively studied. We refer the reader to \cite{BoMa08} where all the arguments needed to prove the following proposition are developed. 
\begin{proposition}
Let $W^Q$ be a colored noise with $Q$ satisfying Assumption \ref{cond_C}. We assume that $u_0$ and $v_0$ are in $L^2(D)$, $\PP$-almost surely. Then, for any time horizon $T$, the system (\ref{FHN_BGT}) has a unique solution $(u,v)$ defined on $[0,T]$ which is $\PP$-almost surely in $\mathcal{C}([0,T],H)\times\mathcal{C}([0,T],H)$. 
\end{proposition}
\noindent The proof of this proposition relies on Itô Formula, see Chapter 1, Section 4.5 of \cite{DaZa92}, and the fact that the functional defined by
\[
f(x)=x(1-x)(x-a),\quad\forall x\in \R
\]
satisfies the inequality
\[
(f(u)-f(v),u-v)\leq\frac{1+a^2-a}{3}\|u-v\|^2,\quad \forall (u,v)\in H\times H,
\]
which implies that the map $f-\frac{1+a^2-a}{3}{\rm Id}$ is dissipative. The local kinetics of system (\ref{FHN_BGT}), that is the dynamics in the absence of spatial derivative, is illustrated in Figure \ref{Fig_FHN_phase}. It describes the dynamics of the system of ODEs
\begin{equation}\label{FHN_ode}
\left\{
\begin{array}{ccl}
{\rm d}\mathfrak{u}&=&[\frac1\e \mathfrak{u}(1-\mathfrak{u})(\mathfrak{u}-a)-\mathfrak{v}]{\rm d}t,\\
\dd \mathfrak{v}&=&[\mathfrak{u}-\mathfrak{v}]\dd t,
\end{array}
\right.
\end{equation}
when the initial condition $(\mathfrak{u}_0,\mathfrak{v}_0)$ is in $[0,1]\times[0,1]$.
\begin{figure}
\begin{center}
\begin{tikzpicture}
\draw plot[domain=-0.25:1,scale=3] (\x,{10*\x*(1-\x)*(\x-0.1)});
\draw plot[domain=-0.25:1,scale=3] (\x,{\x});
\draw[->] (0,0) --node[very near end,below] {$\mathfrak{u(t)}$} (4,0);
\draw[->] (0,0) --node[very near end,left] {$\mathfrak{v(t)}$} (0,4);
\draw[red,->] (0,0).. controls +(0.5,0) and +(0,-0.5)..(2.7,0.8).. controls +(0,0.5) and +(0.5,0)..(2.3,2.4).. controls +(-0.5,0) and +(0.5,0.1)..(-0.2,2.1).. controls +(-0.5,-0.1) and +(-0.1,0.5)..(0,0);
\draw[blue] (0,0) node {$\bullet$} ;
\draw[blue] (0.675,0.675) node {$\bullet$} ;
\draw[blue] (2.6,2.6) node {$\bullet$} ;
\end{tikzpicture}
\caption{Phase portrait with nullclines of system (\ref{FHN_ode}) for $a=0.1$ and $\e=0.1$. The blue points correspond to the three equilibrium points of the system.}\label{Fig_FHN_phase}
\end{center}
\end{figure}

\noindent We explicitly give the numerical scheme used to simulate system (\ref{FHN_BGT}). Let us define the function $k$ given by
\[
k(x)=\frac1\e\left(-x^3+x^2(1+a)\right), \forall x\in \R.
\]
This function corresponds to the non linear parts of the reaction term $f$. We use the following semi-implicit Euler-Maruyama scheme
\begin{equation}\label{scheme_FHN}
\left\{
\begin{array}{ccl}
\frac{u_{n+1}-u_n}{\D t}&=&\kappa\D u_{n+1}-\frac a\e u_{n+1}+k(u_n)-v_{n+1}+\frac{\sigma}{\sqrt\D t}W^Q_{1,n+1},\\
\frac{v_{n+1}-v_n}{\D t}&=&u_{n+1}-v_{n+1},
\end{array}
\right.
\end{equation}
where $(W^Q_{1,n})_{1\leq n\leq N+1}$ is a sequence of independent $Q$-Wiener processes evaluated at time $1$. Let $(H^*,\mathcal{B}(H^*),\tilde\PP)$ be chosen so that the canonical process has the same law as $W^Q_{1,n+1}$ under $\tilde\PP$. Then, for a given $(u_{n},v_n)\in H^1(D)\times H$, the equation
\[
(\frac{1}{\D t}+\frac a\e+\frac{\D t}{1+\D t})u_{n+1}-\kappa\D u_{n+1}=k(u_n)-\frac{1}{1+\D t}v_n+\frac{\sigma}{\sqrt\D t}W^Q_{1,n+1}
\]
has a unique weak solution $u_{n+1}$ in $H^1(D)$, $\tilde\PP$-almost surely. This fact follows from Lax-Milgram Theorem and a measurable selection theorem, see Section 5 of the survey \cite{Wa80}. Therefore, without loss of generality, we may assume in this section that the probability space is $(\mathcal{C}([0,T],H^*),\mathcal{B}(\mathcal{C}([0,T],H^*)),\PP)$ such that under $\PP$, the canonical process has the same law as $W^Q$.
\begin{remark}
In the scheme (\ref{scheme_FHN}), we could have chosen other ways to approximate the reaction term. For instance, it is also possible to work with a completely implicit scheme with $k(u_{n+1})$ instead of $k(u_n)$ in (\ref{scheme_FHN}).
\end{remark}
Let us consider the weak form for the first equation of (\ref{scheme_FHN}). We get,
\begin{equation}\label{scheme_FHN_weak}
\left\{
\begin{array}{ccl}
(\frac{1}{\D t}+\frac a\e+\frac{\D t}{1+\D t})(u_{n+1},\psi)+\kappa(\nabla u_{n+1},\nabla \psi)&=&(k(u_n),\psi)-\frac{1}{1+\D t}(v_n,\psi)+\frac{\sigma}{\sqrt{\D t}}(W^Q_{1,n+1},\psi),\\
v_{n+1}-\frac{\D t}{1+\D t}u_{n+1}&=&\frac{1}{1+\D t}v_{n}
\end{array}
\right.
\end{equation}
for all $\psi\in H^1(D)$. Let $h>0$ and $(\psi_i,1\leq i\leq N_h)$ be the P1 finite element basis defined in Section \ref{sect_noise}. For $n\geq0$, we define the vectors
\[
\vec{u}_n=(u_{n,i})_{1\leq i\leq N_h},\quad \vec{v}_n=(v_{n,i})_{1\leq i\leq N_h},\quad \vec{W}^Q_{n+1}=(W^Q_{1,n+1}(P_i))_{1\leq i\leq N_h},
\]
which are respectively the coordinates of $u_n$, $v_n$ and $W^Q_{1,n+1}$ w.r.t. the basis $(\psi_i,1\leq i\leq N_h)$. We also define the stiffness matrix $A\in\mathcal{M}_{N_h}(\R)$ and the mass matrix $M\in\mathcal{M}_{N_h}(\R)$ by
\[
A_{ij}=(\nabla\psi_i,\nabla\psi_j),\quad M_{ij}=(\psi_i,\psi_j).
\]
System (\ref{scheme_FHN_weak}) can be rewritten as
\begin{eqnarray*}\label{scheme_FHN_final}
\left(
\begin{array}{cc}
(\frac{1}{\D t}+\frac a\e+\frac{\D t}{1+\D t})M+\kappa A&0\\
-\frac{\D t}{1+\D t}I&I
\end{array}
\right)
\left(
\begin{array}{cc}
\vec{u}_{n+1}\\
\vec{v}_{n+1}
\end{array}
\right)
&=&
\left(
\begin{array}{cc}
0&-\frac{1}{1+\D t}M\\
0&\frac{1}{1+\D t}I
\end{array}
\right)
\left(
\begin{array}{cc}
\vec{u}_{n}\\
\vec{v}_{n}
\end{array}
\right)
+\left(
\begin{array}{cc}
G(\vec{u}_{n})\\
0
\end{array}
\right)
\\&+&
\left(
\begin{array}{cc}
\frac{\sigma}{\sqrt\D t}M&0\\
0&0
\end{array}
\right)
\left(
\begin{array}{cc}
\vec{W}^Q_{1,n+1}\\
0\end{array}
\right),
\end{eqnarray*}
where $G(\vec{u}_n)=(k(u_n),\psi_i)_{1\leq i\leq N_h}\in\R^{N_h}$. As for the parabolic stochastic equation considered in Section \ref{sect_heat}, one may expect a numerical strong error for this scheme of order
\begin{equation}\label{err_eq_nl}
\EE(\|(u_t,v_t)-(u_{t,n},v_{t,n})\|^2)^{\frac12}=\text{O}(h+\sqrt{\D t}),
\end{equation}
for $\D t\leq \D t_0$. In (\ref{err_eq_nl})$, (u_{t,n},v_{t,n})_{t\in[0,T]}$ is the interpolation of the discretized point which is piecewise linear in time.\\
\noindent We end up this section with Figure \ref{Fig_FHN} which displays simulations of the stochastic Fitzhugh-Nagumo model (\ref{FHN_BGT}) with zero Neumann boundary conditions on a cardioid domain and zero initial conditions. The kernel of the operator $Q$ is given by equation (\ref{kernel_xi}) for some $\xi>0$. Due to a strong intensity of the noise source ($\sigma=1$), we observe the spontaneous nucleation of a front wave with irregular front propagating throughout the whole domain. 
\begin{figure}
\begin{center}
\begin{tabular}{ccc}
a)\includegraphics[width=3cm]{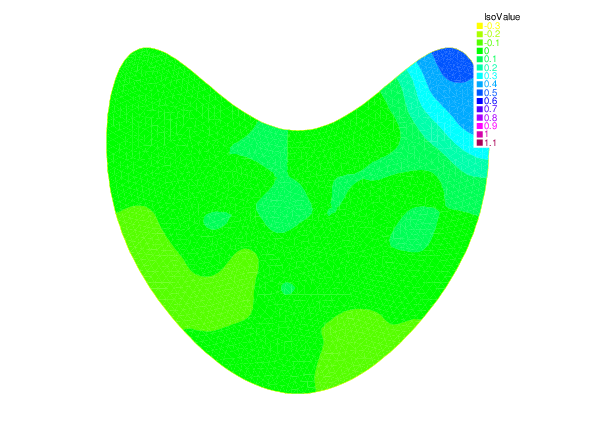}&b)\includegraphics[width=3cm]{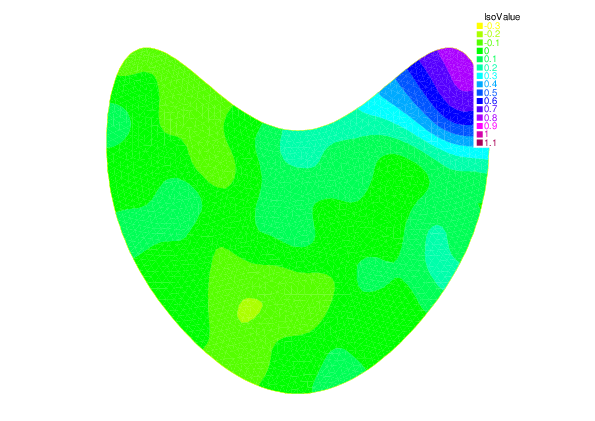}&c)\includegraphics[width=3cm]{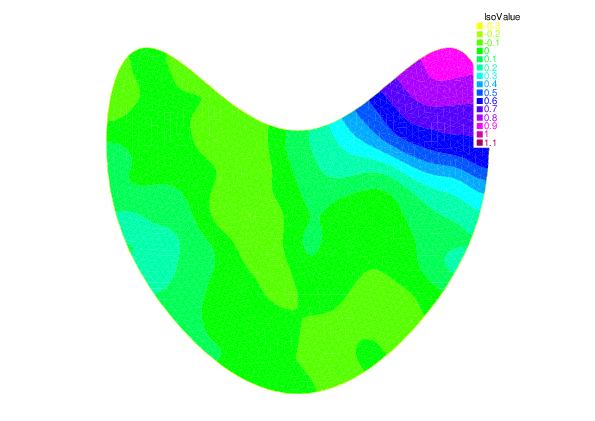}\\ d)\includegraphics[width=3cm]{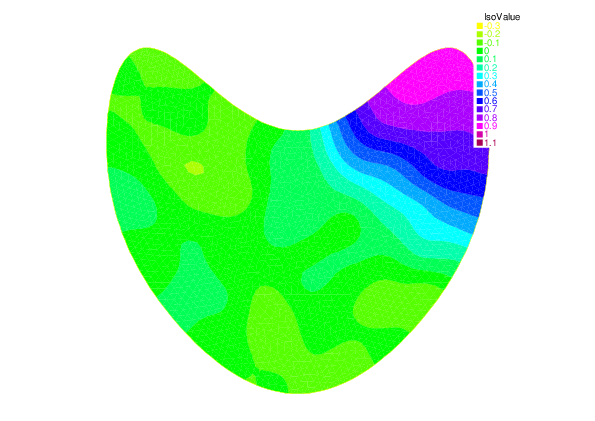}&e)\includegraphics[width=3cm]{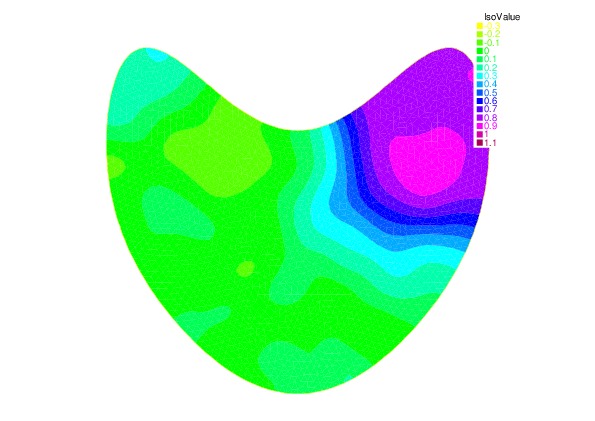}&f)\includegraphics[width=3cm]{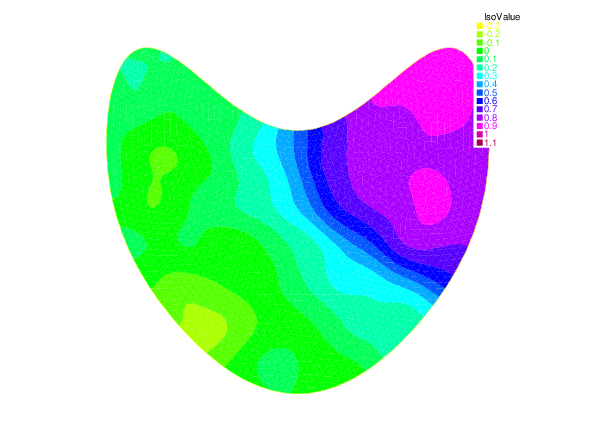}
\end{tabular}
\caption{Simulations of system (\ref{FHN_BGT}) with $\xi=2$, $\sigma=1$, $\varepsilon=0.1$, $a=0.1$. These figures must be read from the up-left to the down-right. The time step is $0.05ms$ and there is $0.5 ms$ between each figure.}\label{Fig_FHN}
\end{center}
\end{figure}

\section{Arrhythmia and reentrant patterns in excitable media}\label{sect_arrhy}

In this section, we focus on classical models for excitable cells, namely Barkley and Mitchell-Schaeffer models. We would like to observe cardiac arrhythmia, that is troubles that may appear in the cardiac beats. Among the diversity of arrhythmia, the phenomena of tachycardia are certainly the most dangerous as they lead to rapid loss of consciousness and death. Tachycardia is described as follows in \cite{JC06}. 
\begin{quote}The vast majority of tachyarrhythmias are perpetuated by reentrant mechanisms. Reentry occurs when previously activated tissue is repeatedly activated by the propagating action potential wave as it reenters the same anatomical region and reactivates it.
\end{quote}
In system (\ref{eq_model}), the equation on $u$ gives the evolution of the cardiac action potential. The equation on $v$ takes into account the evolution of internal biological mechanisms leading to the generation of this action potential. We will be more specifically interested by two systems of this form: the Barkley and Mitchell-Schaeffer models.

\subsection{Numerical study of the Barkley model}

\subsubsection{The model}\label{am}

In the deterministic setting, a paradigm for excitable systems where reentrant phenomena such as spiral, meander or scroll waves have been observed and studied is the Barkley model, see \cite{B90,B91,B92,B94}. This deterministic model is of the following form
\begin{equation}\label{Bar_det}
\left\{
\begin{array}{ccl}
{\rm d}u&=&[\kappa\D u+\frac1\e u(1-u)(u-\frac{v+b}{a})]{\rm d}t,\\
\dd v&=&[u-v]\dd t.
\end{array}
\right.
\end{equation}
The parameter $\e$ is typically small so that the time scale of $u$ is much faster than that of $v$. 
For more details on the dynamic of waves in excitable media, we refer the reader to \cite{K80}. The Barkley model, like two-variables models of this type, faithfully captures the behavior of many excitable systems. The deterministic model (\ref{Bar_det}) does not exhibit re-entrant patterns unless one imposes special conditions on the domain: for instance, one may impose that a portion of the spatial domain is a "dead zone". This means a region with impermeable boundaries where equations (\ref{Bar_det}) do not apply: when a wave reaches this dead region, the tip of the wave may turn around and this induces a spiral behavior, see Section 2.2 of \cite{K80}. One may also impose specific initial conditions such that some zones are intentionally hyper-polarized: the dead region is somehow transient in this case.

\subsubsection{Reentrant patterns}

As in \cite{Sh05} we add a colored noise with kernel of type (\ref{kernel_xi}) to equation (\ref{Bar_det}) and so we consider
\begin{equation}\label{Bar}
\left\{
\begin{array}{ccl}
{\rm d}u&=&[\kappa\D u+\frac1\e u(1-u)(u-\frac{v+b}{a})]{\rm d}t+\sigma \dd W^{Q_\xi},\\
\dd v&=&[u-v]\dd t,
\end{array}
\right.
\end{equation}
where the kernel of $Q_\xi$ is given by (\ref{kernel_xi}) for $\xi>0$.\\
Figure \ref{Fig_Bar_carre} displays a simulation of system (\ref{Bar}) on the square $D=[0,l]\times [0,l]$ with periodic boundary conditions:
\begin{equation}\label{bcond}
\begin{array}{lll}
\forall t\in\R_+,&\forall x\in[0,l]\quad  u_t(x,0)=u_t(x,l),&\quad{\rm and}\quad \frac{\partial u_t}{\partial \vec n}(x,0)=\frac{\partial u_t}{\partial \vec n}(x,l),\\
&\forall y\in[0,l]\quad u_t(0,y)=u_t(l,y),&\quad{\rm and}\quad \frac{\partial u_t}{\partial \vec n}(0,y)=\frac{\partial u_t}{\partial \vec n}(l,y),
\end{array}
\end{equation}
where $\vec n$ is the external unit normal to the boundary. The numerical scheme is based on the following variational formulation. Given $u_0$ and $v_0$ in $H^1(D)$, find $(u_n,v_n)_{1\leq n\leq N}$ such that for all $0\leq n\leq N-1$,
\begin{equation}\label{scheme_Bar_weak}
\left\{
\begin{array}{ccl}
(\frac{u_{n+1}-u_n}{\D t},\psi)+\kappa(\nabla u_{n+1},\nabla \psi)&=&\frac{1}{\e}(u_n(1-u_n)(u_n-\frac{v_n+b}{a}),\psi)+\frac{\sigma}{\sqrt\D t}(W^Q_{1,n+1},\psi),\\
\frac{v_{n+1}-v_n}{\D t}&=&u_{n+1}-v_{n+1}
\end{array}
\right.
\end{equation}
with boundary conditions $u_n(x,0)=u_n(x,l)$, $u_n(0,x)=u_n(l,y)$ and for all $\psi\in H^1(D)$ satisfying $\psi(x,0)=\psi(x,l)$ and $\psi(0,y)=\psi(l,y)$ for any $(x,y)\in [0,l]\times[0,l]$. We have solved this problem using the P1 finite element methods, see Section \ref{section_fe}.\\
Our aim is to observe reentrant patterns generated by the presence of the noise source in this system.
\begin{figure}
\begin{center}
\begin{tabular}{cccc}
a)\includegraphics[width=2cm]{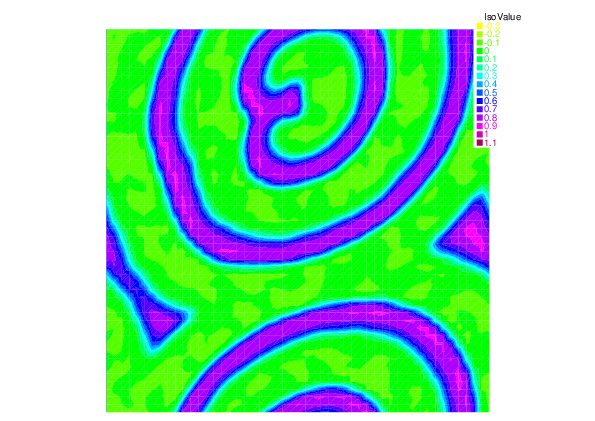}&b)\includegraphics[width=2cm]{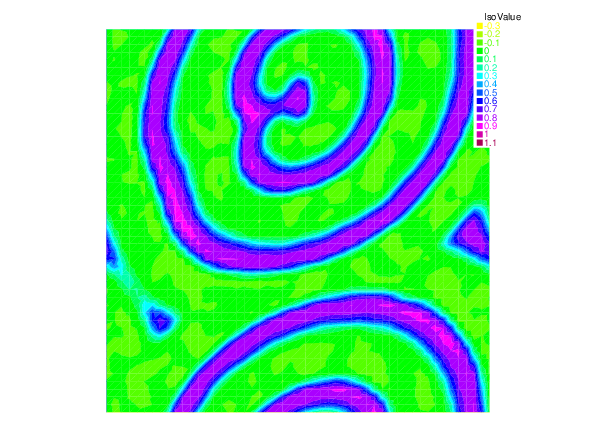}&c)\includegraphics[width=2cm]{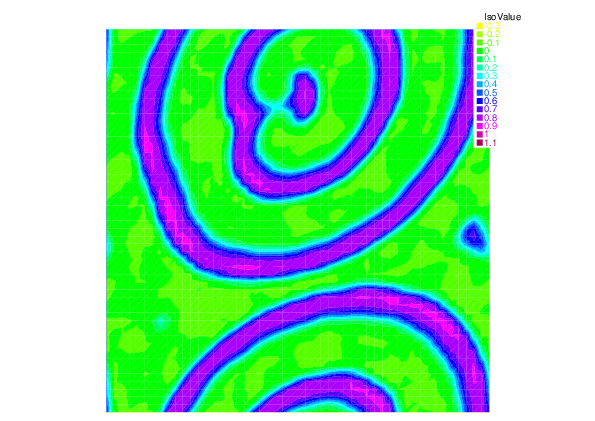}&d)\includegraphics[width=2cm]{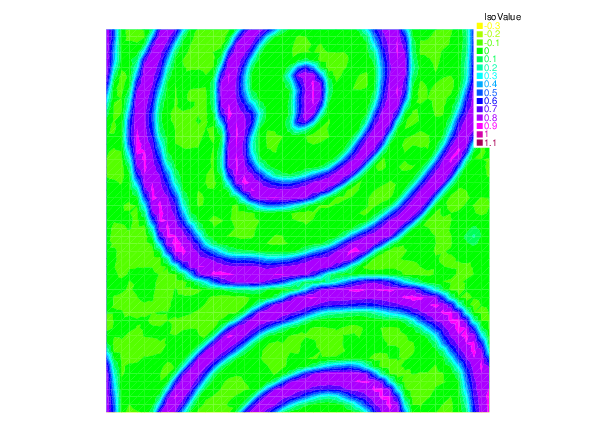}\\
e)\includegraphics[width=2cm]{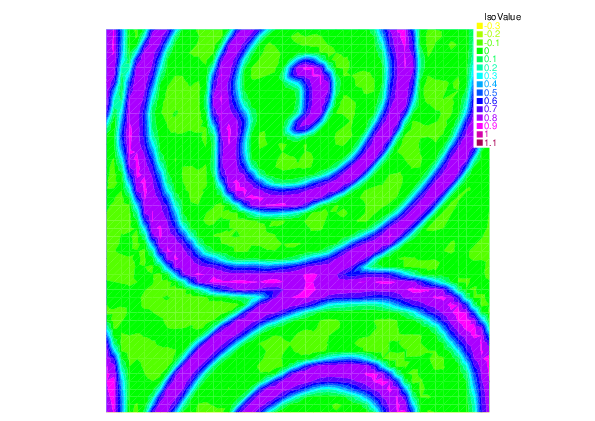}&f)\includegraphics[width=2cm]{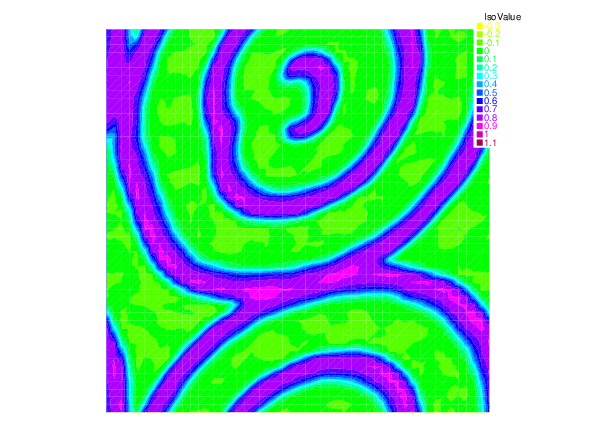}&g)\includegraphics[width=2cm]{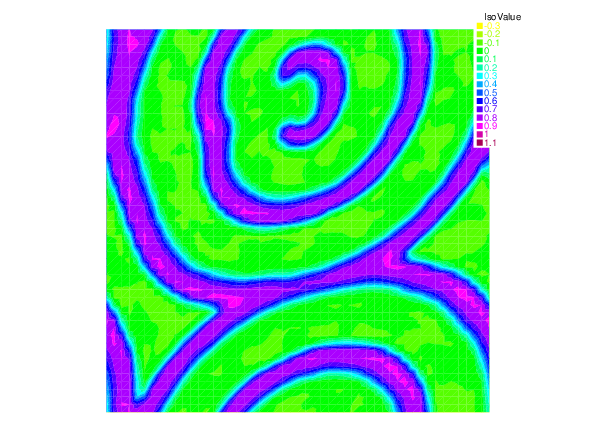}&h)\includegraphics[width=2cm]{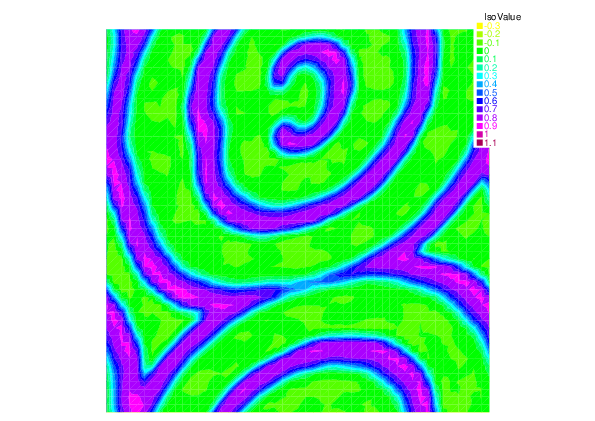}\\
i)\includegraphics[width=2cm]{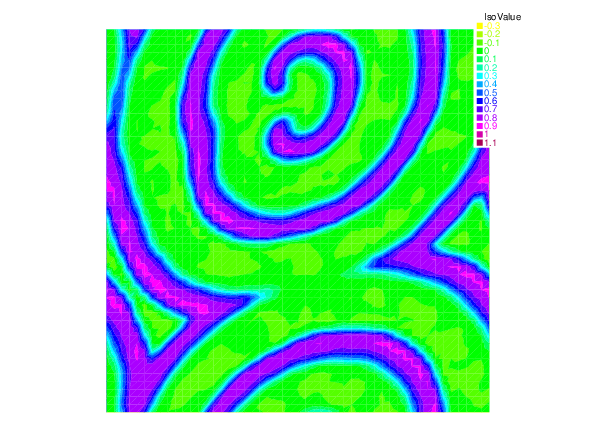}&j)\includegraphics[width=2cm]{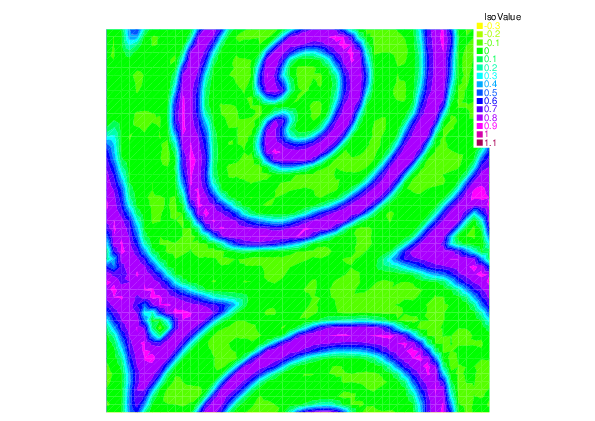}&k)\includegraphics[width=2cm]{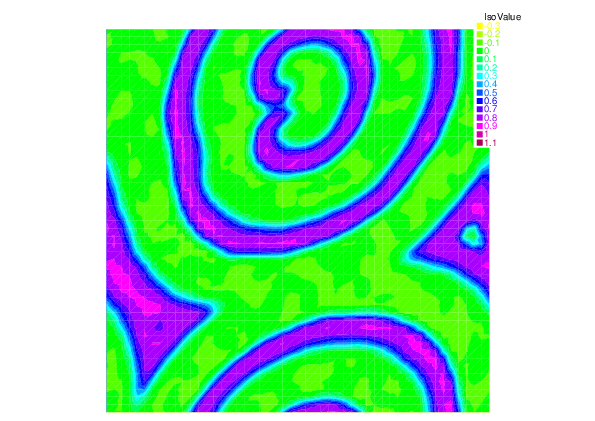}&l)\includegraphics[width=2cm]{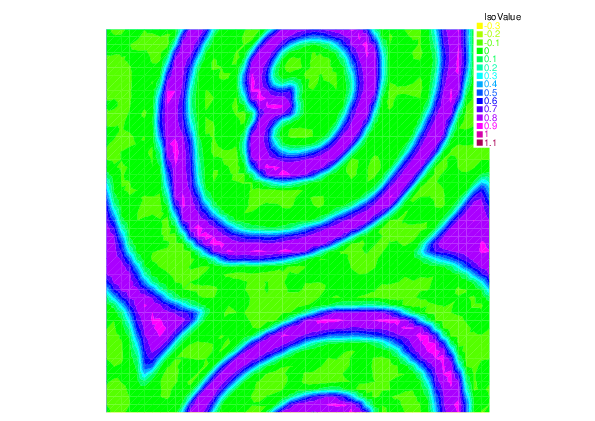}\\
m)\includegraphics[width=2cm]{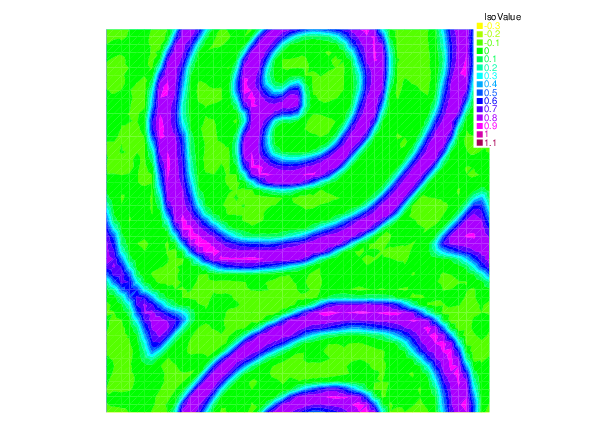}&n)\includegraphics[width=2cm]{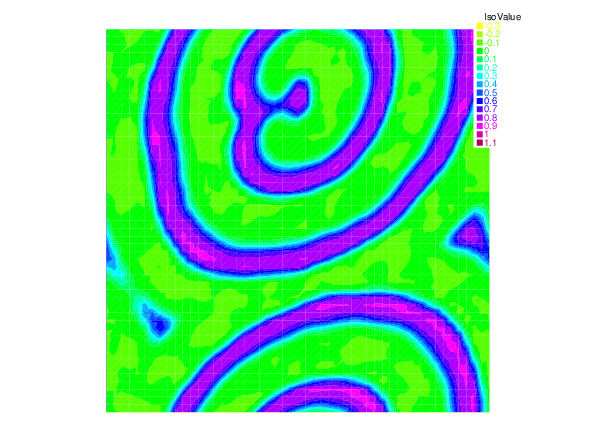}&o)\includegraphics[width=2cm]{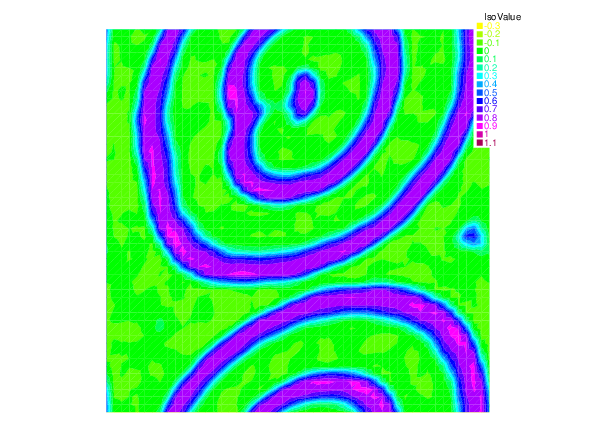}&p)\includegraphics[width=2cm]{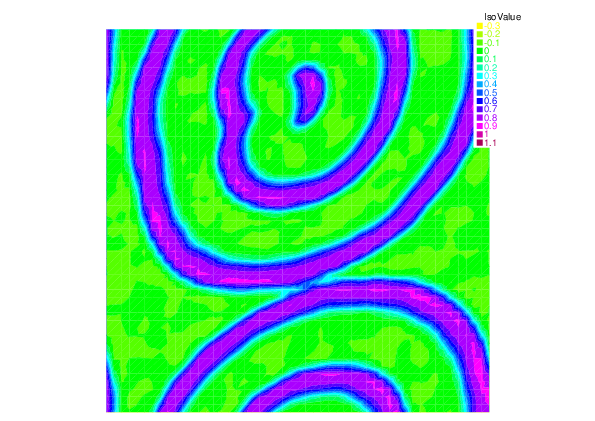}
\end{tabular}
\caption{Reentry is observed for system (\ref{Bar}) with $\xi=2$, $\sigma=0.15$, $\varepsilon=0.05$, $a=0.75$, $b=0.01$ and $\nu=1$. These figures must be read from the top-left to the bottom-right. The quiescent state is represented in green whereas the excited state is in violet. If time is recorded in $ms$, there is $0.5 ms$ between each figures for a time step of $0.05 ms$.}\label{Fig_Bar_carre}
\end{center}
\end{figure}
Figure \ref{Fig_Bar_carre} displays simulations of (\ref{Bar}) using the P1 finite element method. We observe the spontaneous generation of waves with a reentrant pattern. At some points in the spatial domain, the system is excited and exhibits a reentrant evolution which is self-sustained: a previously activated zone is re-activated by the same wave periodically. As explained in \cite{JC06} and quoted in Section \ref{am}, this phenomenon can be interpreted biologically as tachycardia in the heart tissue. We observe that, as in \cite{Sh05}, the constants $a$ and $b$ are chosen such that the deterministic version of system (\ref{Bar}) may exhibit spiral pattern, see the bifurcation diagram between $a$ and $b$ in \cite{B94}. However, in our context, the generation of spiral is a phenomenon which is due solely to the presence of noise. In particular, there is no need for a "dead region", as previously mentioned for the observation of spirals or reentrant patterns in a deterministic context. 
\noindent Figure \ref{Fig_Bar_coeur} displays a simulation of system (\ref{Bar}) on a cardioid domain  with zero Neumann boundary conditions, see (\ref{bound_cond}). We observe the spontaneous generation of a wave turning around itself like a spiral and thus reactivating zones already activated by the same wave.
\begin{figure}
\begin{center}
\begin{tabular}{cccc}
a)\includegraphics[width=2cm]{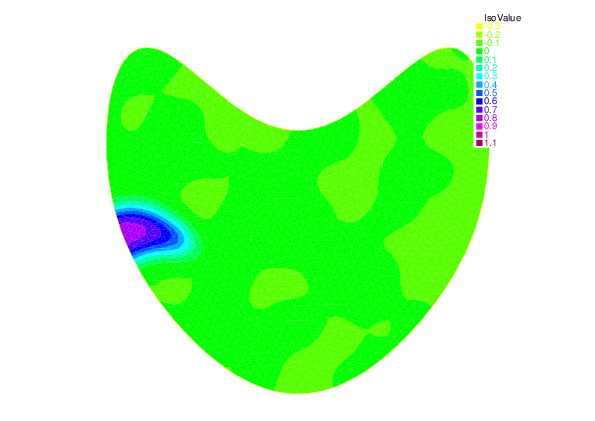}&b)\includegraphics[width=2cm]{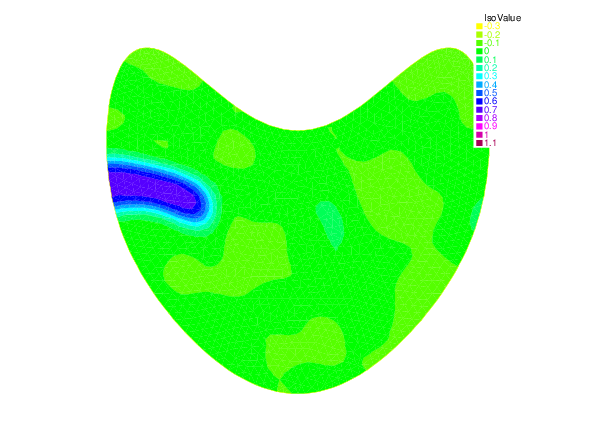}&c)\includegraphics[width=2cm]{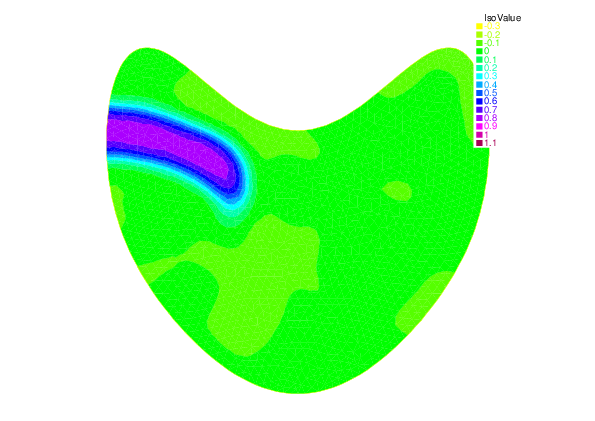}&d)\includegraphics[width=2cm]{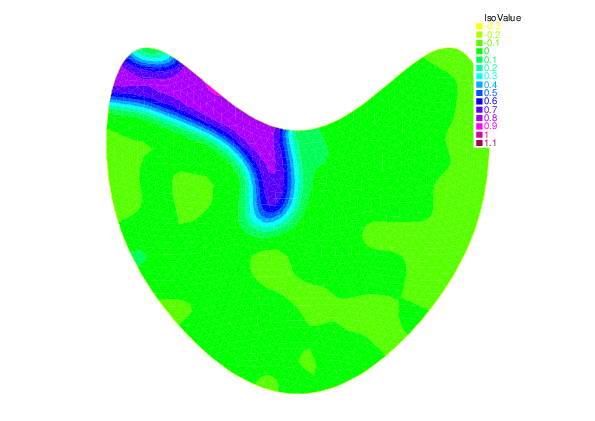}\\
e)\includegraphics[width=2cm]{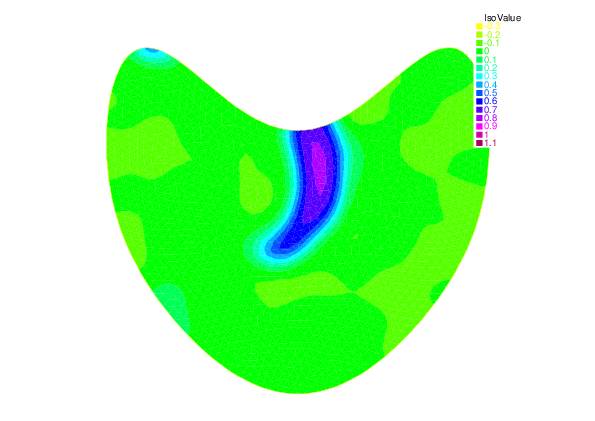}&f)\includegraphics[width=2cm]{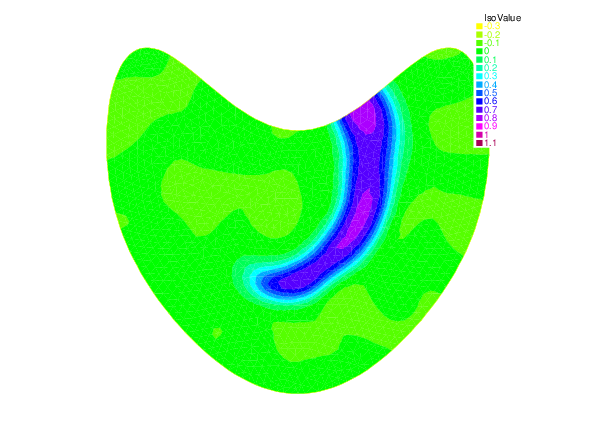}&g)\includegraphics[width=2cm]{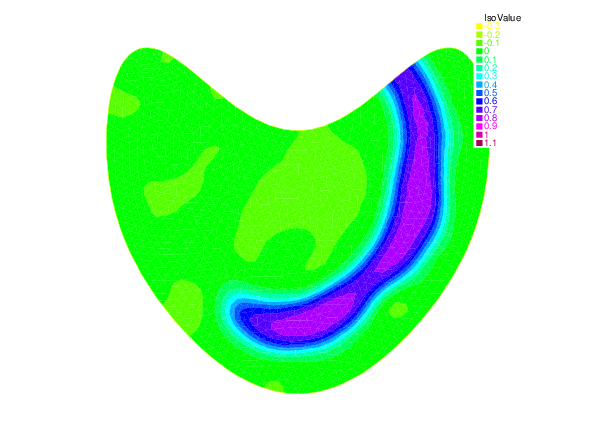}&h)\includegraphics[width=2cm]{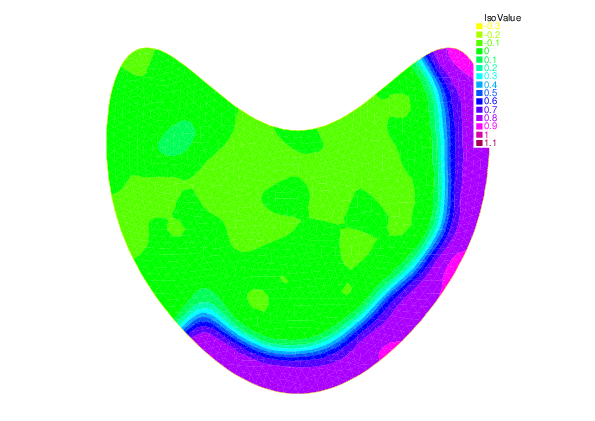}\\
i)\includegraphics[width=2cm]{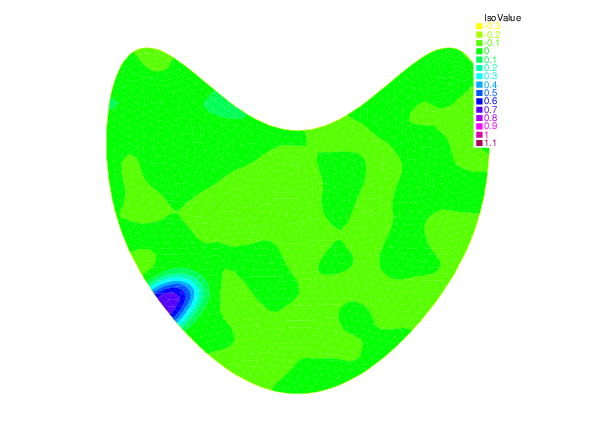}&j)\includegraphics[width=2cm]{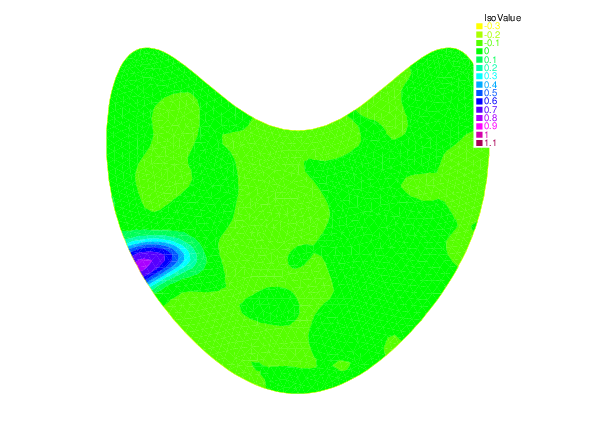}&k)\includegraphics[width=2cm]{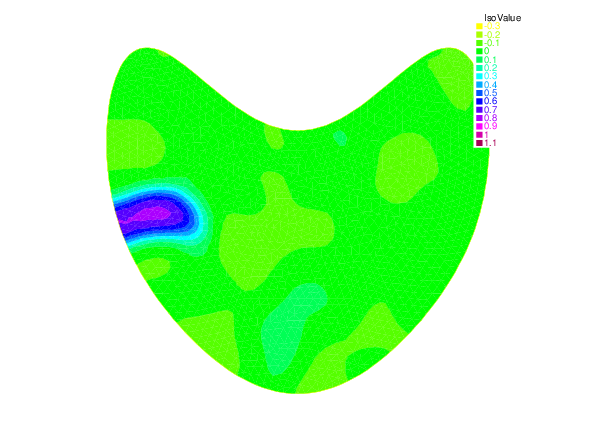}&l)\includegraphics[width=2cm]{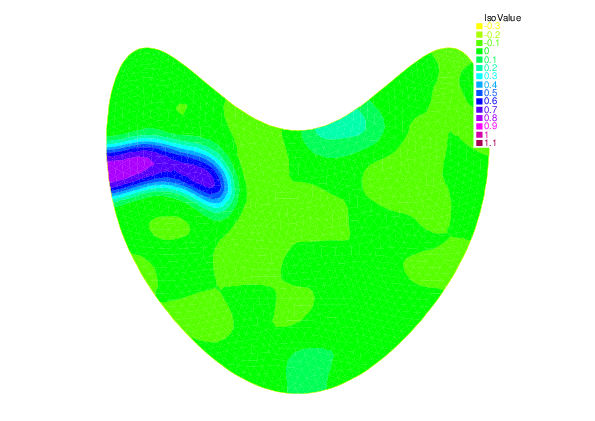}
\end{tabular}
\caption{Simulations of system (\ref{Bar}) with $\xi=2$, $\sigma=0.15$, $\varepsilon=0.05$, $a=0.75$, $b=0.01$ and $\nu=1$. As for the previous figure, the quiescent state is represented in green whereas the excited state is in violet. Another phenomena of re-entry is observed on this cardioid geometry with zero Neumann boundary conditions. There is $2 ms$ between each snapshot for a time step for the simulations equals to $0.05 ms$.}\label{Fig_Bar_coeur}
\end{center}
\end{figure}

To gain a better insight into these reentrant phenomena, a bifurcation diagram between $\e$ and $\sigma$ in system (\ref{Bar}) is displayed in Figure \ref{Fig_Bar_carre_bif}. In this figure, the other parameters $a,b,\nu,\xi$ are held fixed. The domain and boundary conditions are the same as for Figure \ref{Fig_Bar_carre}. Three distinct areas emerge from repeated simulations:
\begin{itemize}
\item the area NW (for No Wave) where no wave is  observed.
\item the area W (for Wave) where at least one wave is generated on average. Such waves do not exhibit reentrant patterns.
\item the area RW (for Reentrant Wave) where waves with re-entry are observed. The wave has the same pattern as in Figure \ref{Fig_Bar_carre}.
\end{itemize} 
Let us mention that these three different area emerge from repeated simulations. The detection of re-entrant patterns is quite empirical here: we say that there is re-entrant patterns if we can actually see it on the figures. An automatic detection of re-entrant patterns may certainly be derived from \cite{LoTh12} even if rather difficult to implement in our setting.\\
At transition between the areas W and RW, ring waves with the same pattern as reentrant waves may be observed: two arms which join each other to form a ring. We also remark that for a fixed $\e$, when $\sigma$ increases, the number of nucleated waves increases. On the contrary, for a fixed $\sigma$ when $\e$ increases the number of nucleated waves decreases.
\begin{figure}
\begin{center}
\begin{tikzpicture}
\begin{scope}[scale=0.3]
\draw[very thin,gray] (0,0) grid (8.5,15.5);
\draw[blue, fill,opacity=0.5] (0.5,15).. controls +(0,-0.5) and +(0,0.5)..(0.5,0.5).. controls +(0.5,-0.5) and +(-0.5,-0.5)..(2,3.5).. controls +(0.5,0.5) and +(-0.5,-0.5)..(3,4.75).. controls +(0.5,0.5) and +(-0.5,-0.5)..(4,6.1).. controls +(0.5,0.5) and +(-0.5,-0.5)..(4.75,6.7).. controls +(0.5,0.5) and +(-0.2,-0.5)..(6,9).. controls +(0.2,0.5) and +(-0.2,-0.5)..(6.5,10).. controls +(0.2,0.5) and +(0,-0.5)..(7,15);
\draw[green, fill,opacity=0.5] (3.75,15).. controls +(0,-0.5) and +(0.1,0.5)..(3.75,12).. controls +(-0.1,-0.5) and +(0.1,0.5)..(3.25,7.5).. controls +(-0.1,-0.5) and +(-0.5,0)..(3.75,6.5).. controls +(0.5,-0.1) and +(-0.5,0)..(4.25,6.5).. controls +(0.5,0.1) and +(-0.2,-0.5)..(6,9).. controls +(0.2,0.5) and +(-0.2,-0.5)..(6.5,10).. controls +(0.2,0.5) and +(0,-0.5)..(7,15);
\draw (1,0) node[below]{$0.01$};
\draw (7,0) node[below]{$0.07$};
\draw (0,1) node[left]{$0.05$};
\draw (0,14) node[left]{$0.18$};
\draw[thick,->] (0,0) -- (0,15.5) node[left]{$\sigma$};
\draw[thick,->] (0,0) -- (8.5,0) node[below]{$\e$};
\draw (5,3) node[fill=white]{NW};
\draw (2,10) node[fill=blue!50]{W};
\draw (5,11) node[fill=blue!32!green]{RW};
\draw (12.5,8.5) node{\includegraphics[width=2cm]{simu_580.png}};
\draw[->,thick] (10,8.5) -- (4.5,8.5);
\draw (-5.5,2) node{\includegraphics[width=2cm]{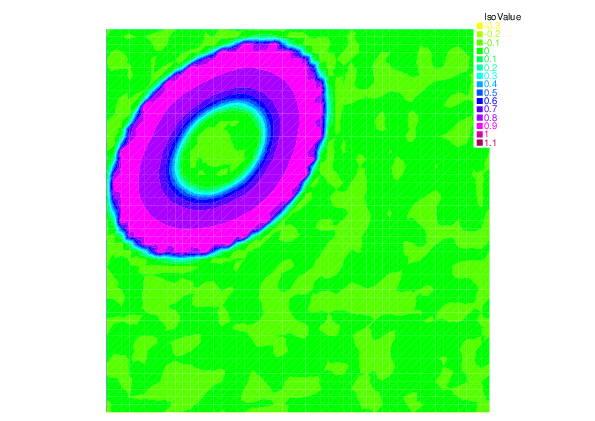}};
\draw[->,thick] (-3,2) -- (1,2);
\draw (12.5,1) node{\includegraphics[width=2cm]{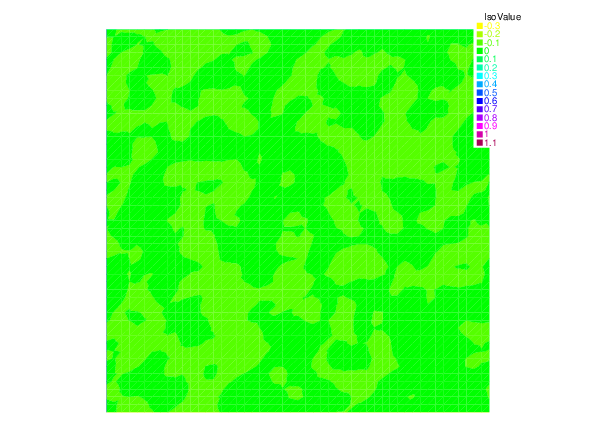}};
\draw[->,thick] (10,1) -- (6,1);
\end{scope}
\end{tikzpicture}
\end{center}
\caption{Numerical bifurcation diagram between $\e$ and $\sigma$ of system (\ref{Bar}) with $\xi=2$, $a=0.75$, $b=0.01$ and $\nu=1$ held fixed. The $\varepsilon$ and $\sigma$-step are respectively $0.005$ and $0.025$ but has been refined around boundaries to draw the curve boundaries.}\label{Fig_Bar_carre_bif}
\end{figure}
Let us notice that for small $\e$, that is when the transition between the quiescent and excited state is very sharp, small noise may powerfully initiate spike. However, we only observe reentrant patterns when $\e$ is large enough. Notice also that the separation curve between the zone NW and the two zones W and RW is exponentially shaped. This may be related to the large deviation theory for slow-fast system of SPDE.

\subsection{Numerical study of the Mitchel-Schaeffer model}

\subsubsection{The model}

Fitzhugh-Nagumo model is the most popular phenomenological model for cardiac cells. However this model has some flaws, in particular the hyperpolarization and the stiff slope in the repolarization phase. The Mitchell-Schaeffer model \cite{MS03} has been proposed to improve the shape of the action potential in cardiac cells. The spatial version of the Mitchell-Schaeffer model reads as follows
\begin{equation}\label{MS}
\left\{
\begin{array}{ccc}
{\rm d}u&=&\left[{\displaystyle \kappa\D u+\frac{v}{\tau_{\rm in}}u^2(1-u)-\frac{u}{\tau_{\rm out}}}\right]{\rm d}t+\sigma \dd W^{Q_\xi},\\
\dd v&=&\left[{\displaystyle\frac{1}{\tau_{\rm open}}\left(1-v\right)1_{u< u_{\rm gate}}-\frac{v}{\tau_{\rm close}}1_{u\geq u_{\rm gate}}}\right]\dd t.
\end{array}
\right.
\end{equation}
The numerical scheme is based on the following variational formulation. Given $u_0$ and $v_0$ in $H^1(D)$, find $(u_n,v_n)$ such that for all $0\leq n\leq N-1$,
\begin{equation}\label{scheme_Bar_weak}
\left\{
\begin{array}{ccl}
(\frac{u_{n+1}-u_n}{\D t},\psi)+\kappa(\nabla u_{n+1},\nabla \psi)&=&\frac{1}{\e}(\frac{v_n}{\tau_{\rm in}}u^2_n(1-u_n)-\frac{1}{\tau_{\rm out}}u_n,\psi)+\frac{\sigma}{\sqrt{\D t}}(W^Q_{n+1}(1),\psi),\\
\frac{v_{n+1}-v_n}{\D t}&=&{\displaystyle \frac{1}{\tau_{\rm open}}\left(1-v_n\right)1_{u_n< u_{\rm gate}}-\frac{v_n}{\tau_{\rm close}}1_{u_n\geq u_{\rm gate}}}
\end{array}
\right.
\end{equation}
for $\psi\in H^1(D)$. More precisely, we solve this problem with the P1 finite element method.

\subsubsection{Numerical investigations}

Bifurcations have been investigated in Figure \ref{Fig_MS_carre_bif} for the same domain and boundary conditions as for the bifurcation diagram related to Barkley model (Figure \ref{Fig_Bar_carre_bif}).  We choose to fix all the parameters except the intensity of the noise $\sigma$ and $\tau_{\rm close}$ to investigate the influence of the strength of the noise and the characteristic time for the recovery variable $v$ to get closed. From repeated simulations, five distinct areas emerge:
\begin{itemize}
\item the area NW (for No Wave) where no wave is observed.
\item the area W (for Wave) where at least one wave is generated on average. These waves do not exhibit reentrant patterns. However, these waves may be generated with the same pattern as reentrant waves: two arms which meet up  and agree to form a ring.
\item the area RW (for Reentrant Wave) where waves with re-entry may be observed as in Figure \ref{Fig_Bar_carre}.
\item the area DW (for Disorganized Wave) where reentrant waves are initiated but break down in numerous pieces resulting in a very disorganized evolution. In a sense, this disorganized evolution may be regarded as reentrant since previously activated zone may be re-activated by one of these resulting pieces.
\item the area T (for Transition) is a transition area between reentrant waves and more disorganized patterns as observed in the area DW.
\end{itemize}

\begin{figure}
\begin{center}
\begin{tikzpicture}
\begin{scope}[scale=0.3]
\draw[very thin,gray] (0,0) grid (12.5,15.5);
\draw (1,0) node[below]{$3.5$};
\draw (10,0) node[below]{$4.4$};
\draw (0,1) node[left]{$0.05$};
\draw (0,6) node[left]{$0.10$};
\draw (0,14) node[left]{$0.18$};
\draw[thick,->] (0,0) -- (0,15.5) node[left]{$\sigma$};
\draw[thick,->] (0,0) -- (12.5,0) node[below]{$\tau_{{\rm close}}$};
\draw[orange,fill,opacity=0.5] (0.5,10.5).. controls +(0,-0.5) and +(-0.5,0)..(2.5,8.5).. controls +(0.5,0) and +(-0.5,0)..(5.5,8.5).. controls +(0.5,0) and +(-0.5,0)..(7.5,7.5).. controls +(0,0.5) and +(-0.5,0)..(7.75,11.5).. controls +(0,0.5) and +(-0.5,-0.5)..(9.5,12).. controls +(0.5,0.5) and +(-0.5,-0.5)..(12.5,13).. controls +(0,0.5) and +(0,-0.5)..(12.5,15.5).. controls +(-0.5,0) and +(0.5,0)..(0.5,15.5).. controls +(0,-0.5) and +(0,0.5)..(0.5,10.5);
\draw[green,fill,opacity=0.5](7.5,7.5).. controls +(0,0.5) and +(-0.5,0)..(7.75,11.5).. controls +(0,0.5) and +(-0.5,-0.5)..(9.5,12).. controls +(0.5,0.5) and +(-0.5,-0.5)..(12.5,13).. controls +(0,-0.5) and +(0,0.5)..(12.5,10.5).. controls +(-0.5,0) and +(0.5,0)..(11,9.5).. controls +(-0.5,0) and +(0.5,0)..(9.75,9.5).. controls +(-0.5,0) and +(0,0.5)..(9.5,6.5).. controls +(0.5,0) and +(0.5,0)..(9.25,6.5).. controls +(-0.5,0) and +(0.5,0)..(7.5,7.5);
\draw[blue,fill,opacity=0.5](12.5,10.5).. controls +(-0.5,0) and +(0.5,0)..(11,9.5).. controls +(-0.5,0) and +(0.5,0)..(9.75,9.5).. controls +(-0.5,0) and +(0,0.5)..(9.5,6.5).. controls +(0.5,0) and +(-0.5,0)..(12.5,6.5).. controls +(0,0.5) and +(0,-0.5)..(12.5,10.5);
\draw[green!50!orange,dashed,very thick](12.5,10.5).. controls +(-0.5,0) and +(0,-0.5)..(8.5,10.5).. controls +(0,0.5) and +(0,-0.5)..(8.5,11.75);
\draw (6,4) node[fill=white]{NW};
\draw (5,12) node[fill=orange!50]{DW};
\draw (8.5,9) node[fill=green!50]{RW};
\draw (11,8) node[fill=blue!50]{W};
\draw (10.5,11.5) node[fill=green!50]{T};
\draw (-4.5,3) node{\includegraphics[width=2cm]{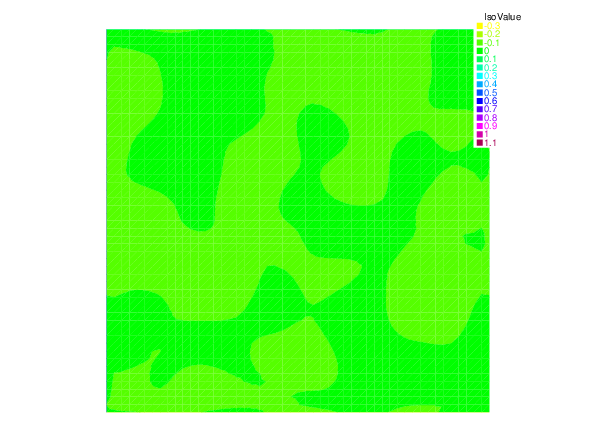}};
\draw[->,thick] (-1.5,3) -- (6,3);
\draw (-4.5,12) node{\includegraphics[width=2cm]{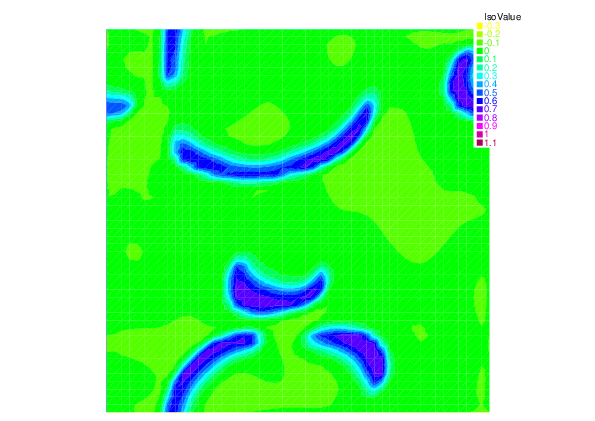}};
\draw[->,thick] (-1.5,12) -- (3,12);
\draw (16.5,9) node{\includegraphics[width=2cm]{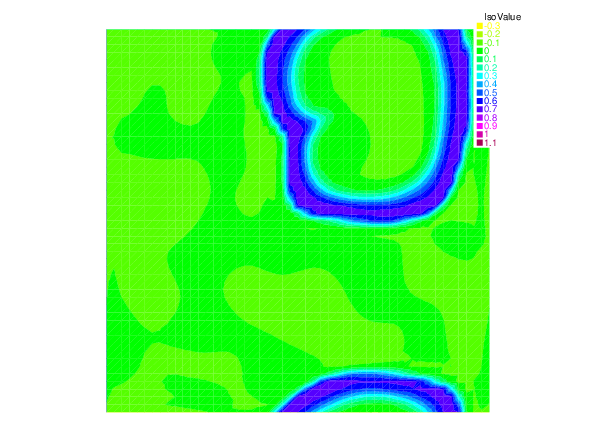}};
\draw[->,thick] (13.5,9) -- (11,9);
\draw (16.5,3) node{\includegraphics[width=2cm]{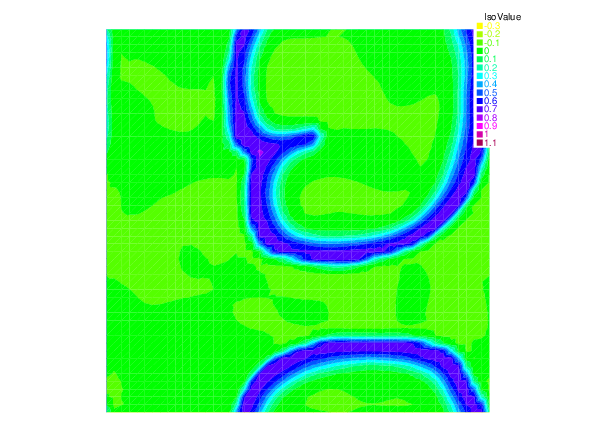}};
\draw[->,thick] (13.5,3) -- (8,8);
\draw (16.5,14) node{\includegraphics[width=2cm]{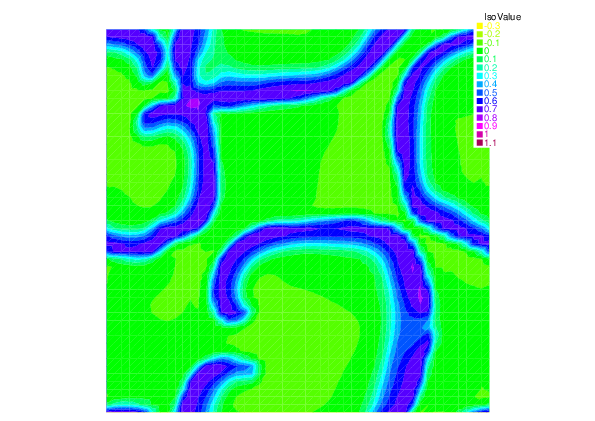}};
\draw[->,thick] (13.5,14) -- (12,12);
\end{scope}
\end{tikzpicture}
\end{center}
\caption{Numerical bifurcation diagram between $\tau_{\rm close}$ and $\sigma$ of system (\ref{MS}) with $\xi=2$, $\tau_{\rm in}=0.07$, $\tau_{\rm out}=0.7$, $\tau_{\rm open}=8.$, $u_{\rm gate}=0.13$ and $\nu=0.03$ held fixed.}\label{Fig_MS_carre_bif}
\end{figure}
We would like to end the paper with a short discussion about a problem we intend to address in future works. In \cite{JT10}, the authors present simulations of a stochastic spatially-extended Hodgkin-Huxley model. This is a celebrated  model of propagation and generation of an action potential in a nerve fiber. They consider the case of a one dimensional nerve fiber stimulated by a noisy signal. They show,  using numerical experiments, that the presence of weak noise in the model may powerfully annihilate the generation (but not the propagation) of waves. We reproduce these numerical experiments in our setting for the Barkley and Fitzhugh Nagumo models in order to investigate if this phenomena may be observed in two dimensions. Thus, we consider the Barkley model with a periodic deterministic input and driven by a colored noise. We perform simulations of this model using growing noise intensity. Unfortunately, we were not able to produce inhibition of a periodic deterministic signal using weak noises. Of course, with stronger intensity of the noise, signals due to the stochasticity of the system are generated and perturb the deterministic periodic signal. However, this phenomenon is not surprising. It seems to us that FitzHugh-Nagumo could be a better model to consider on the question of annihilate the generation of waves by weak noise. We noted that a lot of care is required when choosing the mesh to use as well as the parameters (intensity of the noise, intensity of the deterministic signal, duration of its period), if we want to produce sound results.

\begin{acknowledgements}
The authors would like to thank the anonymous reviewer for his/her valuable comments and suggestions to improve the quality of the paper.
\end{acknowledgements}

\appendix

\section{Proof of Theorem \ref{prop_glob_err}}\label{app_bgt_1}

Recall that the domain $D$ is polyhedral such that
\[
\overline{D}=\bigcup_{T\in\mathcal{T}_h}T.
\]
Let $i\in\{0,0_a,1\}$. The process $(D_h(t),~t\in[0,\tau])$ defined by
\[
D_h(t)=W^Q_t-W^{Q,h,i}_t
\]
is a centered Wiener process. In particular, it is a continuous martingale and thus, by the Burkholder-Davis-Gundy inequality (see Theorem 3.4.9 of \cite{PeZa07}) we have
\[
\EE\left(\sup_{t\in[0,\tau]}\|D_h(t)\|^2\right)\leq c_2\EE(\|D_h(\tau)\|^2)
\]
with $c_2$ a constant which does not depend on $h$ or $\tau$. We begin with the case $i=1$.  Since the processes $W^Q$ and $W^{Q,h,1}$ are regular in space, we write
\[
\EE(\|D_h(\tau)\|^2)=\EE\left(\int_D(W^Q_\tau(x)-W^{Q,h,1}_\tau(x))^2{\rm d}x\right).
\]
We use the definition of $W^{Q,h,1}$ in Definition \ref{Def_app} and the fact that $\sum_{i=1}^{N_h}\psi_i=1$  to obtain
\begin{eqnarray*}
\EE(\|D_h(\tau)\|^2)&=&\EE\left(\int_D(W^Q_\tau(x)-\sum_{i=1}^{N_h}W^Q_\tau(P_i)\psi_i(x))^2{\rm d}x\right)\\
&=&\EE\left(\int_D(\sum_{i=1}^{N_h}(W^Q_\tau(x)-W^Q_\tau(P_i))\psi_i(x))^2{\rm d}x\right)\\
&=&\EE\left(\int_D\sum_{i,j=1}^{N_h}(W^Q_\tau(x)-W^Q_\tau(P_i))(W^Q_\tau(x)-W^Q_\tau(P_j))\psi_i(x)\psi_{j}(x){\rm d}x\right).
\end{eqnarray*}
By an application of Fubini's theorem, exchanging over the expectation, integral and summation, we get
\begin{eqnarray*}
\EE(\|D_h(\tau)\|^2)&=&\sum_{i,j=1}^{N_h}\int_D\EE\left((W^Q_\tau(x)-W^Q_\tau(P_i))(W^Q_\tau(x)-W^Q_\tau(P_j))\right)\psi_i(x)\psi_{j}(x){\rm d}x\\
&=&\tau\sum_{i,j=1}^{N_h}\int_D\left(C(0)-C(P_i-x)-C(P_j-x)+C(P_i-P_j)\right)\psi_i(x)\psi_{j}(x){\rm d}x.
\end{eqnarray*}
For all $1\leq i,j\leq N_h$, if the intersection of the supports of $\psi_i$ and $\psi_j$ is not empty, then
\[
\forall x\in\text{supp} \psi_i,\forall y\in\text{supp}\psi_j, |x-y|\leq K h.
\]
Thus, there exists $K>0$ such that, for all $i,j$, if $\text{supp} \psi_i\cap\text{supp} \psi_{j}\neq \emptyset$ and $x\in \text{supp} \psi_i\cap\text{supp} \psi_{j}$, a Taylor's expansion yields
\[
|C(0)-C(P_i-x)-C(P_j-x)+C(P_i-P_j)|\leq K\max_{x\in\overline{D}}\|{\rm Hess}~C(x)\| h^2,
\]
where we have used the fact that $\nabla C(0)=0$. Then,
\[\EE(\|D_h(\tau)\|^2)\leq K\tau\max_{x\in\overline{D}}\|{\rm Hess}~C(x)\| h^2.
\]
This ends the proof for the case $i=1$. The case $i=0$ can be treated similarly.\\
For the case $i=0_a$, we proceed as follows. The process $W^{Q,h,0_a}$ is the orthonormal projection of $W^Q$ on the space P0, thus, we have, using the Pythagorean theorem,
\[
\EE(\|D_h(\tau)\|^2)=\EE\left(\|W^Q_\tau-W^{Q,h,0_a}_\tau\|^2\right)=\EE\left(\|W^Q_\tau\|^2-\|W^{Q,h,0_a}_\tau\|^2\right).
\]
Then, recalling that the processes $W^Q$ and $W^{Q,h,0_a}$ are regular in space and using the fact that the triangles $T\in\mathcal{T}_h$ do not intersect, we obtain,
\begin{align*}
\EE(\|D_h(\tau)\|^2)&=\EE\left(\int_DW^Q_\tau(x)^2{\rm d} x-\sum_{T\in\mathcal{T}_h}\frac{1}{|T|}(W^Q_\tau,1_T)^2\right).
\end{align*}
By an application of Fubini's theorem, exchanging over the expectation and summation, we get
\begin{equation}\label{bgt_pass}
\EE(\|D_h(\tau)\|^2)=\tau\left(C(0)|D|-\sum_{T\in\mathcal{T}_h}\frac{1}{|T|}(Q1_T,1_T)\right).
\end{equation}
Since $\overline{D}=\bigcup_{T\in\mathcal{T}_h}T$ we have
\[
C(0)|D|=\sum_{T\in\mathcal{T}_h}\frac{1}{|T|}\int_T\int_TC(0){\rm d}z_1{\rm d}z_2,
\]
hence, plugging in (\ref{bgt_pass})
\begin{equation}\label{eq_1}
\EE(\|D_h(\tau)\|^2)=\tau\sum_{T\in\mathcal{T}_h}\frac{1}{|T|}\int_T\int_T[C(0)-C(z_1-z_2)]{\rm d}z_1{\rm d}z_2.
\end{equation}
Thanks to the fact that $\nabla C=0$, a Taylor's expansion yields
\begin{equation}\label{eq_C_p}
C(0)-C(z_1-z_2)=(z_1-z_2)\cdot{\rm Hess}~C(0)(z_1-z_2)+{\rm o}(|z_1-z_2|^2).
\end{equation}
Thus, thanks to  (\ref{hyp_triangle}), for all $z_1,z_2$ in the same triangle $T$
\[
|C(0)-C(z_1-z_2)|\leq K\max_{x\in\overline{D}}\|{\rm Hess}~C(x)\| h^2,
\]
where the constant $K$ is independent from $T\in\mathcal{T}_h$. Plugging in (\ref{eq_1}) yields
\[
\EE(\|D_h(\tau)\|^2)\leq K\max_{x\in\overline{D}}\|{\rm Hess}~C(x)\|\tau h^2
\]
for a deterministic constant $K$.


\end{document}